\documentclass[11pt]{article}
\usepackage{amsfonts}
\usepackage{amsmath,amsthm,amscd,amssymb,mathrsfs,setspace}
\usepackage{latexsym,epsf,epsfig}
\usepackage{color}
\usepackage[hmargin=2.5cm,vmargin=3cm]{geometry}

\newcommand{\ds}{\displaystyle}

\newcommand{\reals}{\mathbb{R}}
\newcommand{\realstwo}{\mathbb{R}^2}
\newcommand{\realsthree}{\mathbb{R}^3}

\newcommand{\xb}{{\bf{x}}}

\newcommand{\cD}{\mathscr{D}}
\newcommand{\Dn}{\partial_{\nu}}

\newcommand{\cE}{{\mathcal{E}}}
\newcommand{\cF}{\mathcal{F}(z)}

\newcommand{\R}{\mathbb{R}}

\newcommand{\bu}{\mathbf u}
\newcommand{\bp}{\mathbf{\Phi}}
\newcommand{\bE}{\mathbf{E}}
\newcommand{\bT}{\mathbb{T}}
\newcommand{\lb}{ \langle}
\newcommand{\rb}{ \rangle}
\theoremstyle{plain}
\newtheorem{theorem}{Theorem}[section]
\newtheorem{lemma}[theorem]{Lemma}

\newtheorem{corollary}[theorem]{Corollary}

\theoremstyle{remark}
\newtheorem{remark}{Remark}[section]
\numberwithin{equation}{section}
\numberwithin{theorem}{section}
\numberwithin{remark}{section}
\numberwithin{assumption}{section}
\numberwithin{condition}{section}

\begin{document}

\title{Eliminating flutter for clamped von Karman plates immersed in subsonic flows}
 \author{\normalsize \begin{tabular}[t]{c@{\extracolsep{.8em}}c}
            {\LARGE Irena Lasiecka} & {\LARGE Justin T. Webster} \\
 \it University of Memphis    &\it North Carolina State University \\
 \it Memphis, TN &\it Raleigh, NC\\[.2cm]
  \it Polish Academy of Sciences, IBS & \it College of Charleston  \\
  \it Warsaw, Poland & \it Charleston, SC\\[.2cm]
    \it lasiecka@memphis.edu &  \it jtwebste@ncsu.edu \\
\end{tabular}}
\maketitle

\begin{abstract}
\noindent We address the long-time behavior of a non-rotational von Karman plate in an inviscid potential flow. The model arises in aeroelasticity and models the interaction between a thin, nonlinear panel and a flow of gas in which it is immersed \cite{bolotin,dowellnon,dowell1}. Recent results in \cite{delay, fereisel} show that the {\em plate component} of the dynamics (in the presence of a physical plate nonlinearity) converge to a global compact attracting set of finite dimension; these results were obtained {\em in the absence of mechanical damping of any type}. Here we show that, by incorporating mechanical damping the full flow-plate system, {\em full trajectories---both plate and flow---converge strongly to (the set of) stationary states}. Weak convergence results, as well as results for smooth initial data, require ``minimal" interior damping, and strong convergence of the dynamics are shown with sufficiently large damping. We require the existence of a ``good" energy balance equation, which is only available when the flows are {\em subsonic}. Our proof is based on first showing the convergence properties for regular solutions, which in turn requires propagation of initial regularity on the infinite horizon. Then, we utilize the exponential decay of the difference of two plate trajectories to show that full flow-plate trajectories are uniform-in-time Hadamard continuous. This allows us to pass convergence properties of smooth initial data to finite energy type initial data. Physically, our results imply that flutter (a non-static end behavior) does not occur in subsonic dynamics.  While such results were known for rotational (compact/regular) plate dynamics \cite{springer} (and references therein),  the result presented herein is the first such result obtained for non-regularized---the most physically relevant---models.
\vskip.15cm

\noindent {\em Key Terms}: Strong stability, nonlinear plates, mathematical aeroelasticity, flutter, nonlinear semigroups
\vskip.15cm
\noindent {\em 2010 AMS}: 74F10, 74K20, 76G25, 35B40, 35G25, 37L15 \end{abstract}

\section{Introduction}

The work herein is motivated by physical problems arising in aeroelasticity  where one of the principal  issues is to eliminate, or reduce,  {\em flutter}  resulting from  the oscillations of thin structures  immersed in  a flow of gas. Canonical examples include: suspended bridges or buildings subjected to strong winds, airfoils oscillating in the flow of gas,  or a fixed  panel   element oscillating in  gas flow, or wind mills in an open field, etc. It is well recognized that flutter is a systemic phenomenon and one of the principal technological challenges is to control  the fluttering behavior of mechanical structures.

From a mathematical point of view the aforementioned flutter problem may be described as {\it stabilization to a given set} of a coupled evolutionary PDE system which consists of a nonlinear plate  interacting in a hybrid way with a three dimensional flow equation.  This is one of the principal  PDE models arising in aeroelasticity: the interactive dynamics between a nonlinear clamped plate and a surrounding potential flow \cite{bolotin,dowellnon}. (This class of models is standard in the applied mathematics literature  and dates back to classical discussions \cite{bolotin,dowellnon}, and also \cite{B,dowell1} and the references therein). Specifically, we consider von Karman plates in the absence of {\em rotational inertia} (consistent with the conventions of aeroelasticity \cite{dowellnon}), where the plate is thin and in-plane accelerations are not accounted for \cite{lagnese}. It is by now well known that the presence of rotational terms provides a regularizing effect on the transverse velocity of the plate, which in turn leads to several desirable mathematical properties such as compactness, gain of derivatives,  etc. Thus, it was natural that the early PDE  theory of flow-plate dynamics  was developed with rotational inertia  present. With this assumption there is a rather rich theory  that has been developed
 for  both well-posedness and long-time behavior of the  dynamics, subjected to strong
 mechanical damping implemented on surface the plate (see  \cite{springer} and the references therein).
 In particular, it is known that  such a system  is ``strongly stable" in the subsonic regime of velocities \cite{springer}.
 However, the  analysis mentioned above depends critically on the presence of rotational inertia terms
 which provide regularizing effects. From the modeling point of view  the appropriate model \cite{dowell1}
 should  be considered  {\it without plate  inertial terms}.  This  task defines  the main challenge of the present paper: to determine conditions under which the resulting  system is {\it strongly stable}---in a sense that full flow-plate
trajectories converge,  in the strong topology of the underlying  finite energy space,  to a set of equilibria. Under simplifying physical assumptions this will imply that flutter is non-present asymptotically in time.
 Our result demonstrates that, indeed, flutter can be eliminated  (in the subsonic regime), provided a sufficiently large damping is applied to the plate.

 On the mathematical side, the primary issue to contend with is low regularity of the hyperbolic Neumann map (from the boundary into the interior)---i.e., the failure of the uniform Lopatinski conditions in dimensions higher than one. This precludes direct analysis of the effects of the coupling on stability by resorting to the compactness properties of the aforementioned Neumann map  (valid in the rotational case, where the velocity of the plate has one additional derivative).
 Instead, in the non-rotational case, there is a loss of $1/3$ derivative  (with respect to finite energy space) in the  Neumann map  \cite{tataru},  which then  prevents  any sort of direct analysis via the coupling.  This challenge  is reflected in our methodology which
 must depart substantially from previous literature.   In fact, while we still rely on many past developments, the key point of departure is that we can no longer afford to treat the problem component-wise. Rather we must rely on global analysis involving so called  relaxed compactness  \cite{dafermos,slemrod},  which depends on  time invariance of higher energies. This latter property is highly non-trivial due to the effects of the physical von Karman nonlinearity on the plate.

\subsection{Previous literature in relation to the present work}\label{prevv}
The study in \cite{delay} provides   long-time asymptotic  properties of finite energy solutions (for the non-rotational plate); the originally ``rough'' dynamics become, {\em without any added damping}, smooth and finite dimensional in nature. To obtain ``compact attracting behavior" in the absence of structural damping mechanisms we note that the flow has the ability to dissipate mechanical plate energy and thereby induce a degree of stability to the structural dynamics. This dissipative effect is not immediately noticeable in the standard energy balance equation. However, a reduction technique introduced in \cite{LBC96,b-c-1}---described below in Theorem \ref{rewrite}---allows us to write the full flow-structure interaction as a certain delayed plate model, and demonstrates the stabilizing effects of the flow provided that {\em rotational inertia terms in the plate are not present}. The flow dynamics manifest themselves in the form of non-conservative forces acting upon the structure via the {\em downwash} of the flow. In the case where rotational inertia is present in the plate model, the downwash of the flow is not substantial enough to dissipate the mass term due to inertia.

We now mention two other closely related scenarios which have been studied in the literature to date: (i) the addition of thermoelastic dynamics to the plate, and (ii) the presence of rotational inertia {\em and} strong mechanical damping. The treatments in \cite{ryz,ryz2} consider the plate ($\alpha\ge0$) with the addition of a heat equation in the plate dynamics. In this case no damping is necessary, as the analytic smoothing and stability properties of thermoelasticity provide ultimate compactness of the finite energy plate dynamics and, furthermore, provide a stabilizing effect to the flow dynamics as well. As for (ii) mentioned above, results on ultimate compactness of plate dynamics, as well as convergence (for subsonic flows) of full flow-plate trajectories to points of equilibria, were announced in \cite{chuey}, with a proof appearing in \cite{springer}.

The main goal of the present paper is to provide a  strengthening of the aforementioned results in \cite{delay} for the case of  {\it subsonic} flow velocities.  Noting that the attracting set  above is obtained only for the structural dynamics (via the utilization of the reduction result), we show that, indeed, the full flow-plate dynamics exhibit strong convergence properties. Specifically, we demonstrate that with damping  on the surface of the plate full flow-plate trajectories converge strongly to the set of stationary points of the dynamical system generated by solutions. We emphasize that our results {\em require only frictional and static damping in the structure} and {\em do not make use of advantageous parabolic effects}. Our result provides further physical insight to the panel flutter problem; in \cite{delay} the flow, while driving the plate dynamics, also contributes a stabilizing effect to the plate dynamics as well. Our present result indicates that for panels in subsonic flow, strong stability of the plate (via mechanical damping) can be {\em transferred} to the flow, in some sense. This is in agreement with experimental and numerical studies wherein {\em divergence} (or `buckling') of panels is observed for subsonic flow velocities, in contrast to chaotic or periodic behavior ({\em i.e., flutter}) in the case of supersonic flow velocities.
We quote from a recent survey article authored by E. Dowell \cite{dowellrecent} with regard to subsonic flows: ``...if the trailing edge, as well as the leading edge, of the panel is fixed then divergence (a static aeroelastic instability) will occur rather than flutter. Panel divergence is a form of aeroelastic buckling..."
In our analysis below, the subsonic nature of the flow is critical, as it provides a viable energy balance equation. This energy balance is not presently available for weak solutions to \eqref{flowplate} when $U>1$, and the energy identity is necessarily polluted by nondissipative terms. In fact, this very issue is what kept  well-posedness of energy solutions to the supersonic flow-plate interaction open until very recently  \cite{supersonic}.

 In view of this, the distinct feature of our work is to utilize the results of \cite{delay,fereisel} to show that if sufficiently large damping is considered in the plate, the full flow-plate trajectories converge to the set of stationary points of the flow-plate system. We require a ``good" energy relation (present in the case of subsonic dynamics) which provides finiteness of the {\em dissipation integral}).
 Our overall approach is very much informed by the earlier work in \cite{springer,ryz,ryz2} (and references therein), however there are multiple technical hurdles which prevent these older considerations from being applicable. In order to utilize the earlier work (and account for the loss of regularity of the plate velocity) we must operate on regularized flow-plate trajectories. After showing the desired convergence properties for regular trajectories, we must then pass these (via an approximation argument) to energy level initial data in the state space.
\subsection{Notation}
For the remainder of the text we write $\xb$ for $(x,y,z) \in \realsthree_+$ or $(x,y) \in \Omega \subset \realstwo_{\{(x,y)\}}$, as dictated by context. Norms $\|\cdot\|$ are taken to be $L_2(D)$ for the domain $D$. The symbols $\nu$ and $\tau$ will be used to denote the unit normal and unit tangent vectors to a given domain, again, dictated by context. Inner products in $L_2(\realsthree_+)$ are written $(\cdot,\cdot)$, while inner products in $L_2(\R^2\equiv\partial\R^3_+)$ are written $\left<\cdot,\cdot\right>$. Also, $ H^s(D)$ will denote the Sobolev space of order $s$, defined on a domain $D$, and $H^s_0(D)$ denotes the closure of $C_0^{\infty}(D)$ in the $H^s(D)$ norm
which we denote by $\|\cdot\|_{H^s(D)}$ or $\|\cdot\|_{s,D}$.  When $s =0 $ we may abbreviate the notation
to $\| \cdot \| $. We make use of the standard notation for the trace of functions defined on $\realsthree_+$, i.e., for $\phi \in H^1(\realsthree_+)$, $tr[\phi]=\phi \big|_{z=0}$ is the trace of $\phi$ on the plane $\{\xb:z=0\}$. (We use analogous notation for $tr[w]$ as the trace map from $\Omega$ to $\partial \Omega$.)

\subsection{PDE Description of the model}

The gas flow environment we consider is $\realsthree_+=\{(x,y,z): z > 0\}$. The plate
 is  immersed in an inviscid  flow (over body) with velocity $U \neq 1$ in the  $x$-direction. (Here we normalize $U=1$ to be Mach 1, i.e., $0 \le U <1$ is subsonic and $U>1$ is supersonic.)
The plate is modeled by  a bounded domain $\Omega \subset \reals^2_{\{(x,y)\}}=\{(x,y,z): z = 0\}$ with smooth boundary $\partial \Omega = \Gamma$ and
the scalar function $u: \Omega \times \R_+ \to \reals$ represents the transverse displacement of the plate in the $z$-direction at the point $(x,y)$ at the moment $t$.
We focus on the panel configuration as it is physically relevant and friendly to mathematical exposition. Current work is being undertaken to investigate the model presented below in the case of other plate and flow boundary conditions \cite{fereisel,websterlasiecka}.

We then  consider  the following general form of  plate equation with internal  nonlinear forcing $f(u)$, external excitation $p(x,t)$, and potential damping
$ku_t$ with nonnegative  coefficients $k$ and $\beta$:
\begin{equation}\label{plate}\begin{cases}
u_{tt}+\Delta^2u+ku_t+\beta u+f(u)= p(\xb,t) ~~ \text { in } ~\Omega\times (0,T), \\
u=\Dn u = 0  ~~\text{ on } ~ \partial\Omega\times (0,T),  \\
u(0)=u_0,~~u_t(0)=u_1.
\end{cases}
\end{equation}

\noindent \textbf{Nonlinearity}:
\noindent We consider the von Karman nonlinearity, based upon the assumptions of {\em finite elasticity} and {\em maintained orthogonality of the plate filaments} \cite{lagnese}:
\begin{equation}\label{karman}
 f(u)= f_V(u)=-[u, v(u)+F_0],
\end{equation}
where $F_0$ is a given in-plane load,
the von Karman bracket $[u,v]$  is given by:
\begin{equation*}
[u,v] = \partial_{xx} u\partial_{yy} v +
\partial_{yy} u\partial_{xx} v -
2\partial_{xy} u\partial_{xy}
v,
\end{equation*} and
the Airy stress function $v(u)$ is defined  by the relation $v(u)=v(u,u)$ where  $v(u,w)$
 solves the following  elliptic problem
\begin{equation}\label{airy-1}
\Delta^2 v(u,w)+[u,w] =0 ~~{\rm in}~~  \Omega,\quad \Dn v(u,w)= v(u,w) =0 ~~{\rm on}~~  \partial\Omega,
\end{equation}
for given $u,w\in H^2_0(\Omega)$.
\begin{remark}
Another nonlinearity of interest is that of {\em Berger's} nonlinearity $f_V$. This is a valid physical approximation of the von Karman nonlinearity when the panel is clamped or hinged, and takes the form
$$f_B(u) = [\Upsilon-\kappa ||\nabla u||^2]\Delta u,$$ for $\kappa \ge 0$ and $\Upsilon$ bounded below, both physical parameters. The Berger nonlinearity satisfies all of the same key estimates as the von Karma nonlinearity, and, as such, the results we present here also hold when $f_V$ is replaced by $f_B$.
\end{remark}
\vskip.2cm
\noindent \textbf{Damping}: Specific assumptions will be imposed on the size of the damping
and will  depend on the type of result to be obtained. These will be given later.

In full generality, the term $k u_t$ could be replaced by $k(\xb)g(u_t)$ for $k(\xb)\in L_{\infty}(\Omega)$ and $g \in C(\reals)$ some monotone damping function with further constraints, as in \cite{springer, Lu}.

\begin{remark}\label{inertia}
We pause here to mention that in many investigations of nonlinear plates so called {\em rotational inertia} in the filaments of the plate is taken into account; this effect is encompassed in a term $-\alpha\Delta u_{tt}$, $\alpha >0$ appearing in the LHS of the plate equation, where $\alpha$ is proportional to the thickness of the plate squared. When $\alpha =0$ we are referring to the non-rotational plate, as described above. This rotational term is mathematically advantageous, as its presence has a regularizing effect on the plate velocity at the level of finite energy; this in turn produces an additional measure of compactness into the model. The von Karman nonlinearity discussed herein acts {\em compactly on the finite energy space} for $\alpha>0$. Here, we seek to prove convergence results for the dynamics without this  term---this constitutes the ``appropriate" model for panel flutter \cite{dowell1}.
\end{remark}

For the flow component of the model, we make use of linear
 potential theory \cite{bolotin,dowell1} and the (perturbed) flow potential $\phi:\realsthree_+ \rightarrow \reals$ which satisfies the equation below:
\begin{equation}\label{flow}\begin{cases}
(\partial_t+U\partial_x)^2\phi=\Delta \phi & \text { in } \realsthree_+ \times (0,T),\\
\phi(0)=\phi_0;~~\phi_t(0)=\phi_1,\\
\Dn \phi = d(\xb,t)& \text{ on } \realstwo_{\{(x,y)\}} \times (0,T).
\end{cases}
\end{equation}
The strong coupling here takes place in the  downwash term of the flow potential (the Neumann boundary condition) by taking $$d(\xb,t)=-\big[(\partial_t+U\partial_x)u (\xb)\big]\cdot \mathbf{1}_{\Omega}(\xb),~~~
\xb\in\R^2,
$$
and by taking  the aerodynamical pressure in \eqref{plate}
of the form
\begin{equation}\label{aero-dyn-pr}
p(\xb,t)=p_0(\xb)+\big(\partial_t+U\partial_x\big)tr[\phi].
\end{equation}
Above, $\mathbf{1}_{\Omega}(\xb)$ denotes the indicator function of $\Omega$ in $\R^2$.
 This structure of $d(\xb,t)$  corresponds  to the case when the part of the  boundary $z=0$
 outside of the plate is
 the surface of a  rigid body.

This gives the fully coupled model:
\begin{equation}\label{flowplate}\begin{cases}
u_{tt}+\Delta^2u+ku_t+\beta u+f(u)= p_0+\big(\partial_t+U\partial_x\big)tr[\phi] & \text { in } \Omega\times (0,T),\\
u(0)=u_0;~~u_t(0)=u_1,\\
u=\Dn u = 0 & \text{ on } \partial\Omega\times (0,T),\\
(\partial_t+U\partial_x)^2\phi=\Delta \phi & \text { in } \realsthree_+ \times (0,T),\\
\phi(0)=\phi_0;~~\phi_t(0)=\phi_1,\\
\Dn \phi = -\big[(\partial_t+U\partial_x)u (\xb)\big]\cdot \mathbf{1}_{\Omega}(\xb) & \text{ on } \realstwo_{\{(x,y)\}} \times (0,T).
\end{cases}
\end{equation}

\subsection{Energies, state space, and well-posedness}\label{energies0}
As we are considering the subsonic case $U \in [0,1)$, we may resort to known theory
with  standard velocity multipliers and boundary conditions to derive the energy. This procedure leads to an energy which is bounded from below, given in Lemma \ref{energybound}. See \cite{springer,jadea12,webster} for more details.
In this case, we have  {\it the flow} and {\it interactive } energies given, respectively, by
\begin{align*}
 E_{fl}(t) = & \dfrac{1}{2}\big[\|\phi_t\|_{\R^3_+}^2-U^2\|\partial_x\phi\|_{\R^3_+}^2+\|\nabla \phi\|_{\R^3_+}^2\big],~~
 E_{int}(t)=2U\lb tr[\phi],u_x\rb_{\Omega}
\end{align*}
The  {\it plate energy} is defined as usual:
\begin{align}\label{energies}
E_{pl}(t) =& \dfrac{1}{2}\big[\|u_t\|_{\Omega}^2+ \|\Delta u\|_{\Omega}^2 \big] +\Pi(u)+\dfrac{\beta}{2}\|u\|^2,
\end{align}
\noindent $\Pi(u)$ is a potential of the nonlinear and nonconservative forces  given
by
 \begin{equation}
 \Pi(u)=\Pi_V(u)=\dfrac{1}{4}\|\Delta v(u)\|_{\Omega}^2-\dfrac{1}{2} \lb[u,u],F_0\rb_{\Omega}-\lb p,u\rb_{\Omega}
 \end{equation}
corresponding to the von Karman nonlinearity. Additionally, we note that the {\em static damping} produces the conserved quantity $\beta \|u\|^2$, with $\beta$ a parameter that will be determined at specific points in the arguments below.

The total energy  then is defined as a sum of the three components $${\cE}(t) = E_{fl}(t) + E_{pl}(t) +  E_{int}(t)  $$ and satisfies
 \begin{equation}\label{enident} \ds {\cE}(t)+\int_s^t\int_{\Omega} k |u_t|^2 d\Omega d\tau = {\cE}(s).\end{equation} The above quantities provide us with the finite energy considerations for the model, and hence the appropriate functional setup for well-posedness of weak and strong solutions. In short, we may say that weak solutions  satisfy the variational relations associated to \eqref{flowplate}. Generalized solutions are, by definition, strong limits of strong solutions to \eqref{flowplate}; however, in practice, we rely on the theory of semigroups and in this way generalized solutions are viewed as semigroup solutions to \eqref{flowplate}. Generalized solutions then satisfy an integral formulation of \eqref{flowplate} and are called {\em mild} by some authors. We now provide the technical definition of solutions which will be needed below:

A pair of functions $\big(u(x,y;t), \phi(x,y,z;t)\big)$ such that
\begin{equation}\label{platereq}
u(\xb,t) \in C([0,T]; H_0^2(\Omega))\cap C^1([0,T];L_2(\Omega)),\end{equation}
\begin{equation}\label{flowreq}
\phi(\xb,t) \in C([0,T]; H^1(\realsthree_+))\cap C^1([0,T];L_2(\realsthree_+)) \end{equation} is said to be a strong solution to \eqref{flowplate} on $[0,T]$ if
\begin{itemize}
\item $(\phi_t;u_t) \in L_1(a,b; H^1(\realsthree_+)\times H_0^2(\Omega))$ for any $(a,b) \subset [0,T]$.
\item $(\phi_{tt};u_{tt}) \in L_1(a,b; L_2(\realsthree_+)\times L_2(\Omega))$ for any $(a,b) \subset [0,T]$.
\item $\phi(t)\in H^2(\realsthree_+)$ and $\Delta^2 u(t) \in L_2(\Omega)$ for almost all $t\in [0,T]$.
\item The equation $$ u_{tt}+\Delta^2 u+ku_t +\beta u+f(u)=p(\xb,t)$$ holds in $L_2(\Omega)$ for almost all $t >0$.
\item The equation $$(\partial_t + U\partial_x)^2\phi=\Delta \phi$$ holds in $L_2(\realsthree_+)$ for almost all $t>0$ and almost all $\xb \in \realsthree_+$.
\item The boundary conditions in \eqref{flowplate} hold for almost all $t\in [0,T]$ and for almost all $\xb \in \partial \Omega$, $\xb \in \realstwo$ respectively.
\item The initial conditions are satisfied pointwise; that is $$\phi(0)=\phi_0, ~~\phi_t(0)=\phi_1, ~~u(0)=u_0, ~~u_t(0)=u_1.$$
\end{itemize}

As stated above, generalized solutions are strong limits of strong solutions; these solutions will correspond to semigroup solutions for an initial datum outside of the domain of the generator.
\vskip.25cm
\noindent\textbf{Generalized solutions}

\noindent A pair of functions $\big(u(x,y;t), \phi(x,y,z;t)\big)$ is said to be a generalized solution of the problem \eqref{flowplate} on the interval $[0,T]$ if (\ref{platereq}) and (\ref{flowreq}) are satisfied and there exists a sequence of strong solutions $(\phi_n(t);u_n(t))$ such that
$$\lim_{n\to \infty} \max_{t \in [0,T]} \Big\{\|\partial_t\phi-\partial_t \phi_n(t)\|_{L_2(\realsthree_+)}+\|\phi(t)-\phi_n(t)\|_{H^1(\realsthree_+)}\Big\}=0$$ and
$$\lim_{n \to \infty} \max_{t \in [0,T]} \Big\{\|\partial_t u(t)-\partial_t u_n(t)\|_{L_2(\Omega)} + \|u(t) - u_n(t)\|_{H_0^2(\Omega)}\Big\}=0.$$

 Owing to the natural requirements on the functions above, we name our state space $$Y=Y_{pl}\times Y_{fl} = H_0^2(\Omega)\times L_2(\Omega)\times H^1(\realsthree_+)\times L_2(\realsthree_+).$$ Additionally, due to the structure of the spatial operator in the flow equation we topologize $Y_{fl}$ with the seminorm (corresponding to Section \ref{energies0}) $$\|(\phi_0,\phi_1)\|_{Y_{fl}}^2 = \|\nabla \phi_0\|_{L_2(\realsthree_+)}^2+\|\phi_1\|^2_{L_2(\realsthree_+)}.$$
On the finite time horizon, working with semigroup solutions, the seminorm on $Y_{fl}$ recovers the full $H^1(\realsthree_+)$ norm, owing to the hyperbolic structure of the flow equation. Indeed,
 \begin{equation}\label{hypertime}\|\phi\|_{L_2(\realsthree)} \le \|\phi_0\|_{L_2(\realsthree_+)} + \int_0^T \|\phi_t(\tau)\|_{L_2(\Omega)}d\tau.\end{equation} The previous considerations \cite{jadea12,supersonic,webster} have made critical use of this fact to obtain {\em invariance} with respect to the topology of the energy space $Y$ (associated to semigroup generation on $Y$), where we have taken the gradient seminorm on $H^1(\realsthree_+)$.
 As we are working with convergence of flow solutions in this treatment, the topology of the flow space will be a paramount consideration. As evidenced above, the $L_2(\realsthree_+)$ norm is not controlled by the gradient norm (see Remark \ref{flownorm} below).

 We assume for well-posedness that $0\le U <1$ (subsonic), $k\ge 0$, $\beta \ge 0$, and $p_0 \in L_2(\Omega)$ with $F_0 \in H^{4}(\Omega)$.
\begin{theorem}[{\bf Nonlinear semigroup}]
\label{nonlinearsolution} For all $~T>0$, \eqref{flowplate} has a unique strong (resp. generalized---and hence variational \cite{jadea12, webster}) solution on $[0,T]$  denoted by $S_t (y_0) $. (In the case of strong solutions, the natural compatibility condition must be in force on the data
$ \Dn \phi_0 = -\mathbf{1}_{\Omega}(u_1+Uu_{0x})$.)

 This is to say that $(S_t, Y) $ is a (nonlinear) dynamical
system on $Y$.
 Both strong and generalized  solutions satisfy \begin{equation}\label{eident}\ds {\mathcal{E}}(t)+k\int_s^t \int_{\Omega} |u_t(\tau)|^2 d\Omega d\tau= {\mathcal{E}}(s)\end{equation} for $t>s$.
Moreover, this solution is uniformly bounded in time in the norm of the state space $Y$.
This means that there exists a constant $C$ such that for all $ t \geq 0 $ we have $$ \|S_t (y_0)\|_Y \leq C \left(\|y_0\|_Y\right).$$
\end{theorem}
A detailed proof of the well-posedness results above utilizing semigroup theory, along with discussion,  can be found in \cite{webster}; more recently, a complete study of the trace regularity of solutions can be found in \cite{jadea12} (along with a  proof of well-posedness which makes use of a viscosity approach).

In order to describe the dynamics of the flow in the context of a long time behavior (Remark \ref{flownorm} below), it is necessary to introduce local space for the flow denoted by $\widetilde Y_{fl}$. Convergence in this space is given by convergence with respect to
$$\|(\phi_0,\phi_1)\|_{Y_{fl},\rho}\equiv \int_{ K_{\rho} } |\nabla \phi_0|^2  + |\phi_1|^2 d\xb,$$ {\em for all} $\rho >0$,
where $K_{\rho} \equiv \{ \xb \in R^3_{+}; |\xb |\leq \rho \} $.
By virtue of the Hardy inequality \cite[p.301]{springer}
$$\|\phi_0\|^2_{L_2(K_{\rho})} \le C_{\rho}\|\nabla \phi_0\|^2_{L_2(\realsthree_+)}$$ and hence
$$\|(\phi_0,\phi_1)\|_{Y_{fl},\rho}=  \|(\phi_0,\phi_1)\|_{H^1(K_{\rho})\times L_2(K_{\rho})}^2\le \|(\phi_0,\phi_1)\|_{Y_{fl}}^2.$$ 
\vskip.1cm
By the boundedness in Theorem \ref{nonlinearsolution}, the topology corresponding to $\widetilde Y_{fl}$ (i.e., in $Y_{fl,\rho}$ for any $\rho>0$) becomes a viable measure of long time behavior which we will refer to as the {\em local energy sense} and it is appropriate to take limits of the form ~$\displaystyle \lim_{t \to \infty} \|S_t(y_0)\|_{Y_\rho},~~$ where $\|\cdot \|_{Y_\rho} \equiv \|\cdot \|_{Y_{pl} \times Y_{fl,\rho}}$ and $S_t$ is the flow associated to the well-posedness in Theorem \ref{nonlinearsolution} above.

\begin{remark}\label{flownorm} We pause to further summarize and clarify the relation between the flow topologies: $H^1(\realsthree_+) \times L_2(\realsthree_+)$, $Y_{fl}$, and $Y_{fl,\rho}$. In all analyses, initial flow data is chosen in $H^1(\realsthree_+) \times L_2(\realsthree_+)$. For the well-posedness proof in \cite{webster} semigroup theory is utilized to show generation of a perturbed problem in the the topologies of $Y_{fl}$. Then, via the estimate in \eqref{hypertime}, generation of the original flow-plate problem can be recovered on $H^1(\realsthree_+) \times L_2(\realsthree_+)$ on any $[0,T]$. Solutions are global-in-time bounded in the topology of $Y_{fl}$ but not necessarily global-in-time bounded in the full $H^1(\realsthree_+) \times L_2(\realsthree_+)$ norm (owing to the contribution of the flow component). Hence, for considerations involving convergence of the flow $(\phi, \phi_t)$ as $t \to \infty$, we require a localized perspective, such that we have boundedness in the flow energy topology $Y_{fl}$ and {\em local compactness results}; when restricting to any ball $K_{\rho} \subset \realsthree_+$, this is the case. We also note that if a sequence $y_m \to y_0$ in $Y$ then $y_m \to y_0$ in $\widetilde Y$. \end{remark}

We also introduce the overall dynamics operator $\bT:\mathscr{D}(\bT)\subset Y \to Y$; in our setup, $\bT$ is the generator of the nonlinear semigroup $S_t(\cdot)$. For the sake of exposition we do not give the full structure of this operator (which involves introducing the spatial flow operator on $\realsthree_+$ and the corresponding Neumann map). Rather, we give reference to \cite{webster,jadea12,supersonic} for the details of the abstract setup of the problem. We suffice to say that, via semigroup methods presented in those references, Ball's method provides the generator of the nonlinear semigroup with appropriate (dense) domain $\cD(\bT)$. The key property necessary in this treatment is that \begin{equation}\label{domainprop}\cD(\bT) \hookrightarrow (H^4\cap H_0^2)(\Omega)\times H_0^2(\Omega) \times H^2(\realsthree_+)\times H^1(\realsthree_+). \end{equation} Specifically, the discussion of the generator of the {\em linear dynamics} is described in \cite[p. 3129]{webster}; the  contribution of the von Karman nonlinearity is inert in the characterization of the nonlinear generator $\bT$  due to {\it sharp regularity of Airy's stress function} (see the regularity properties of stationary solutions to the von Karman equations \cite[Theorem 1.5.7]{springer}).
(We are again careful to note that invariance of $\mathscr D(\mathbb T)$ under the semigroup is with respect to the topology of $Y$ or $\widetilde Y$ on the finite time horizon. However, in the limit as $t \to \infty$ we consider the topology of $\widetilde Y$.)

The final result we will need corresponds to the boundedness (from below) of the nonlinear energy.  These bounds are necessary to obtain the boundedness of the semigroup quoted in Theorem \ref{nonlinearsolution} above. We will denote the positive part of the energy $\cE$ as follows:
\begin{equation}
\mathbf{E}(t)\! =\! \frac{1}{2}\left\{\|u_t\|_{\Omega}^2\!+\!\|\Delta u\|_{\Omega}^2\!+\!\beta\|u\|^2\!+\!\|\nabla \phi\|_{\realsthree_+}^2\!-\!U^2\|\partial_x\phi\|^2\!+\!\|\phi_t\|_{\realsthree_+}^2 \right\} +\frac{1}{4}\|\Delta v(u)\|_{\Omega}^2  \end{equation}
First, we have the following bound \cite[Lemma 5.2, p.3136]{webster} :
\begin{lemma}\label{energybound}
For generalized solutions to \eqref{flowplate}, there exist positive constants $c,C,$ and $M$ positive such that
\begin{equation}  c \mathbf E(t)-M_{p_0,F_0} \le \cE(t) \le C \mathbf E(t)+M_{p_0,F_0}  \end{equation}
\end{lemma}
To obtain the above bounds we note that the interactive energy $E_{int}$ has the bound: \begin{equation} |E_{int}(t)| \le U\delta \|\nabla \phi\|_{\realsthree_+}^2+\frac{U}{\delta}\|u_x\|_{\Omega}^2, ~~\delta>0, \end{equation} which follows from the Hardy inequality (see, for instance, \cite[p. 301]{springer}). In the linear case ($f(u)=0$), taking $U$ sufficiently small will provide the boundedness in Lemma \ref{energybound} above. In the nonlinear case, where $f(u)$ is present, the lower bound for the energy given in the lemma depends on the boundary condition imposed on the plate.
The argument depends on the maximum principle for the Monge-Ampere equation and it employs the fact that $u$ is zero on the boundary. The precise statement of the lemma which controls lower frequencies is given \cite{springer}:
\begin{lemma}\label{l:epsilon}
For any $u \in H^2(\Omega) \cap H_0^1(\Omega) $ and   $\epsilon > 0 $ there exists $M_{\epsilon} $ such that
$$\|u\|^2 \leq \epsilon [\|\Delta u\|^2  + \|\Delta v(u) \|^2 ] + M_{\epsilon} $$
\end{lemma}

This accommodates both clamped and hinged boundary conditions imposed on the plate. See the discussion in \cite[Lemma 1.5.5]{springer} for more details. Utilizing this lemma, and the energy identity \eqref{eident}, is how one arrives at the global-in-time energy bound
\begin{lemma}\label{globalbound}
Any generalized (and hence weak) solution to \eqref{flowplate}, with $f(u)=f_V(u)$ satisfies the bound \begin{equation}
\sup_{t \ge 0} \left\{\|u_t\|_{\Omega}^2+\|\Delta u\|_{\Omega}^2+\|\phi_t\|_{\realsthree_+}^2+\|\nabla \phi\|_{\realsthree_+}^2 \right\}  \leq C\big(\|S_t(y_0)\|_Y\big)< + \infty.\end{equation}
\end{lemma}
The above inequalities are proven and discussed in detail in \cite{springer, webster, jadea12}.
Finally, as a corollary to the energy identity for generalized solutions \eqref{eident} and Lemma \ref{energybound}, we obtain the following:

\begin{corollary}\label{dissint}
Let $ k> 0 $. Then the dissipation integral is finite. Namely, for a generalized solution to \eqref{flowplate} we have $$  \int_0^{\infty} \|u_t(t)\|_{0,\Omega}^2 dt \leq  K_u < \infty.$$
\end{corollary}
\noindent This boundedness will be used critically in multiple places to obtain convergence of trajectories to stationary points.

\subsection{Stationary problem}
We now briefly state and discuss the stationary problem associated to \eqref{flowplate}, which has the form:
\begin{equation}\label{static}
\begin{cases}
\Delta^2u+\beta u+f(u)=p_0(\xb)+U\partial_x \phi & x \in \Omega\\
u=\Dn u= 0 & x \in \Gamma\\
\Delta \phi -U^2 \partial_x^2\phi=0 & x \in \realsthree_+\\
\Dn \phi = -\mathbf{1}_{\Omega}(\xb)\cdot U\partial_x u  & x \in \partial \realsthree_+
\end{cases}
\end{equation}
This problem has been studied before in the long-time behavior considerations for flow-plate interactions, most recently in \cite[Section 6.5.6]{springer}; in this reference, the following theorem is shown for subsonic flows (this is given as \cite[Theorem 6.5.10]{springer}):
\begin{theorem}\label{statictheorem}
Suppose $0 \le U <1$, $ k \geq 0, \beta \geq 0 $  and $f(u)=f_V(u)$, with $p_0 \in L_2(\Omega)$ and $F_0 \in H^{4}(\Omega)$. Then {\em weak} solutions $\left(u(\xb),\phi(\xb)\right)$ to \eqref{static} exist and satisfy the additional regularity property $$(u,\phi) \in (H^4\cap H_0^2)(\Omega) \times W_2(\realsthree_+)$$ where $W_k(\realsthree_+)$ is the homogeneous Sobolev space of the form $$W_k(\realsthree_+) \equiv \left\{\phi(\xb) \in L_2^{loc}(\realsthree_+) ~:~ \|\phi\|^2_{W_k}\equiv \sum_{j=0}^{k-1}\|\nabla \phi \|^2_{j,\realsthree_+} \right\}.$$
Moreover, the stationary solutions mentioned above correspond to the extreme points of the  `potential' energy functional $$D(u,\phi) = \frac{1}{2}\|\Delta u\|_{\Omega}^2+\beta\|u\|^2+\Pi(u)+\frac{1}{2}\|\nabla \phi\|_{\realsthree_+}^2-\dfrac{U^2}{2}\|\partial_x\phi\|_{\realsthree_+}^2 + U\lb\partial_x u, tr[\phi]\rb_{\Omega}$$ considered for $(u,\phi) \in H_0^2(\Omega) \times W_1(\realsthree_+)$. \end{theorem} \noindent
The potential energy $D(u, \phi) $ is smooth on a local space  $H_0^2(\Omega) \times W_1(\mathbb R^3_{+})$ and, moreover,
$$D(u,\phi) \geq c [\|\Delta  u\|^2_{\Omega}  + \|\nabla \phi\|_{\realsthree_+}^2 - C.$$
This latter property is a consequence of Lemma \ref{l:epsilon} which controls lower frequencies.
Thus,  it achieves its minimum and the extremal set of the functional $D$ is non-empty.

\noindent We denote the set of all stationary solutions (weak solutions to \eqref{static} above) as $\mathcal N$.

\section{Statement of main result and discussion}\label{mainresults}

\subsection{Main results}
The first main result deals with {\em regular} initial data in the domain of the generator, $\mathscr D(\bT)$.
\begin{theorem}\label{regresult}
Let $0\le U<1$ and let $f(u)=f_V(u)$ and assume $p_0 \in L_2(\Omega)$ and $F_0 \in H^{4}(\Omega)$. Then for all $ k>0$  and $\beta \ge 0$ any  solution $(u(t),u_t(t);\phi(t),\phi_t(t))=S_t(y_0)$ to the flow-plate system \eqref{flowplate} with $y_0 \in
\cD(\bT)$ and spatially localized initial flow data (i.e., there exists a $\rho_0>0$ so that for $|\xb|\ge \rho_0$ we have $\phi_0(\xb) = \phi_1(\xb)=0$) has the property that \begin{align*}\lim_{t \to \infty} \inf_{(\hat u,\hat \phi) \in \mathcal N}\left\{\|u(t)-\hat u\|^2_{H^2(\Omega)}+\|u_t(t)\|^2_{L_2(\Omega)}+\|\phi(t)-\hat \phi\|_{H^1( K_{\rho} )}^2+\|\phi_t(t)\|^2_{L_2( K_{\rho} )} \right\}\\=0, \text{ for any }  \rho>0.\end{align*} \end{theorem}
\begin{remark} The conclusion of Theorem \ref{regresult} is implied by the statement that for any initial data $y_0 \in Y$ and any sequence of times $t_n \to \infty$ there exists a subsequence $t_{n_k} \to \infty$ such that the discrete trajectory $S_{t_{n_k}}(y_0)$ converges strongly (in $Y_{\rho}$) to an element of $\mathcal N$. This is how we shall prove the theorem above corresponding to smooth initial data. \end{remark}
Our next result deals with finite energy  initial data.
\begin{theorem}\label{weakth}Let $0\le U<1$ and let $f(u)=f_V(u)$ and suppose the damping $k>0$, $\beta \geq 0$  in \eqref{flowplate}; assume $p_0 \in L_2(\Omega)$ and $F_0 \in H^{4}(\Omega)$. Also, suppose $y_0 \in Y$ with localized initial flow data (as in Theorem \ref{regresult} above). Then any  solution $(u(t),u_t(t);\phi(t),\phi_t(t))=S_t(y_0)$ to the flow-plate system \eqref{flowplate} with $y_0 \in  {Y}
$ converges weakly in $Y_{\rho}$ (i.e., weakly in $Y_{pl}\times Y_{fl,\rho}$) for any $\rho$, as $t \rightarrow \infty$, to the stationary set $\mathcal{N} $.
\end{theorem}
The above  results  hold with no restriction on damping---i.e., for any $k>0$ and $\beta\geq 0$.
In order to improve weak convergence to a strong convergence for finite energy initial data a stronger influence of the plate damping mechanism is required. The first theorem below makes no assumptions on the structure of the set of stationary solutions to the flow-plate problem \eqref{static}.
\begin{theorem}\label{conequil1}
Let $0\le U<1$ and let $f(u)=f_V(u)$ and assume $p_0 \in L_2(\Omega)$ and $F_0 \in H^{4}(\Omega)$. Then there are minimal damping coefficients $k_0$ and $\beta_0$ so that for $k\ge k_0>0$ and $\beta\ge \beta_0>0$  any generalized solution $(u(t),\phi(t))$ to the system with localized (in space) initial flow data (see Theorem \ref{regresult}) has the property that \begin{align*}\lim_{t \to \infty} \inf_{(\hat u,\hat \phi) \in \mathcal N}\left\{\|u(t)-\hat u\|^2_{H^2(\Omega)}+\|u_t(t)\|^2_{L_2(\Omega)}+\|\phi(t)-\hat \phi\|_{H^1( K_{\rho} )}^2+\|\phi_t(t)\|^2_{L_2( K_{\rho} )} \right\}\\=0, \text{ for any }  \rho>0.\end{align*} \end{theorem}
\begin{remark} The minimal damping coefficients $k_0$ and $\beta_0$ depend on an invariant set for the plate dynamics, which itself depends on the loading $p_0$ and $F_0$, as well as the domain $\Omega$ and the constants $U,\rho_0$, but {\em is independent on the particular initial data of the system}. \end{remark}

If we make a further physical assumption that  $\mathcal N$ is an isolated set (e.g., finite), we have the following second main theorem as a corollary:
\begin{corollary}\label{improve} Assume that $\mathcal N$ is an isolated set.
Let the hypotheses of Theorem \ref{weakth} be in force; then for any generalized solution $(u,\phi)$ to \eqref{flowplate} (with localized flow data, as above), there exists a solution satisfying \eqref{static} such that $$\big(u(t),u_t(t);\phi(t),\phi_t(t)\big) \rightharpoonup (\hat u, 0;\hat \phi, 0),~~t \to \infty$$ with weak convergence taken with respect to the topology of $Y_{\rho}$.

Let the hypotheses of Theorem \ref{conequil1} (strong convegence) be in force; then for any generalized solution $(u,\phi)$ to \eqref{flowplate} (with localized flow data, as above), there exists a solution $(\hat u, \hat \phi)$ satisfying \eqref{static} such that
\begin{align*}\lim_{t \to \infty} \left\{\|u(t)-\hat u\|^2_{H^2(\Omega)}+\|u_t(t)\|^2_{L_2(\Omega)}+\|\phi(t)-\hat \phi\|_{H^1( K_{\rho} )}^2+\|\phi_t(t)\|^2_{L_2( K_{\rho} )} \right\}\\=0, \text{ for any }  \rho>0.\end{align*}

\end{corollary}

\begin{remark}\label{sard}
 For given loads $F_0$ and $p_0$, the set of stationary solution is generically finite. This is to say that there is an open dense set $\mathcal R \subset L_2(\Omega)\times H^4(\Omega)$ such that  if $(p_0,F_0) \in \mathcal R$ then the corresponding set of stationary solutions $\mathcal N$ is finite. This follows from the Sard-Smale theorem, as shown in \cite[Theorem 1.5.7 and Remark 6.5.11]{springer}.
\end{remark}
We pause to reflect on the physical meaning of the above results. First, we note that by the analysis in \cite{delay} no imposed damping is necessary in the system in order to obtain convergence of the plate dynamics to a compact, finite dimensional attracting set. Beyond this, when minimal damping is active in the interior of $\Omega$, we see that the end behavior of full trajectories---plate and flow---is static, but with convergence in a weak sense OR with smooth initial data. When the damping is sufficiently large, the convergence to stationary states---for finite energy initial data---is in the strong topology. This indicates the the subsonic flow-plate dynamics ultimately converge to a static deformation---what is known in aeroelasticity as {\em divergence}. This result corroborates what has been observed physically and numerically: namely, {\em flutter can be eliminated by the introduction of damping  in panels for subsonic flow velocities}.

Several comments are in order:
\begin{enumerate}
\item
The results stated above depend on the presence of nonlinearity in the structure. This has a mitigating effect on controlling low frequencies---see Lemma \ref{l:epsilon}. The analogous result is no longer valid for a {\it  linearization} of the model.
\item
On the other hand, the ability to contend with the nonlinear estimates arising in this analysis is due to the validity of sharp estimates for Airy's stress function. See \cite[p.44]{springer} (and references therein).

\item Our main results are analogous to previous results obtained in \cite{springer} for $\alpha>0$ and in \cite{ryz2} with $\alpha=0$ and a regularizing thermal effect considered. However, due to the failure of the uniform Lopatinski condition for the  Hyperbolic Neumann map, the proof of the present result resorts to very different approach. The proofs of our main theorems require a novel approach which brings together the efforts in \cite{chuey, springer,ryz2} along with new techniques and estimates {\em in the context of delay plate dynamics} and  {\it relaxed}  limit sets \cite{ball, slemrod}. This will be elaborated upon with more technical details  as we proceed with the proof.

\end{enumerate}

\subsection{Discussion}

We now specifically address the difficulties involved in showing a stabilization to equilibria result without assuming either (i) $\alpha>0$ and strong damping of the form $-k_2\Delta u_t$ (as in \cite{chuey,springer}), or (ii) taking $\alpha \ge 0$ and exploiting parabolic effects in a thermoelastic plate \cite{ryz,ryz2}. In both cases, as well as that considered in this treatment, the key task is {\em to first show compact end behavior for the plate dynamics}. This requires the use of a reduction of the full flow-plate dynamics to a delayed plate equation (Theorem \ref{rewrite} below), at which point one may work on this delayed system. In both case (i) and (ii)  above the ultimate character of the nonlinear component of the model is {\em compact}---owing to the fact that parabolic smoothing {\em and} rotational inertia both provide $\nabla u_t \in L_2(\Omega)$. The results in \cite{delay} were the first to show that dissipation could be harvested from the flow in the principal case $\alpha=0$ (via the reduction result) in order to show ultimate compactness of the plate dynamics without imposed mechanical damping nor thermoelastic effects.

In this treatment, the major contribution is the ability to circumvent the seeming lack of natural compactness in the dynamics (particularly in the plate velocity $u_t$). Specifically, the methods which are utilized in showing the (analogous) stabilization to equilibria result in \cite{springer,ryz,ryz2} (and references therein) {\em each critically use that $u_t \to 0$ in $H^1(\Omega)$}. This measure of compactness for the plate component is translated (albeit in different ways) to the flow component of the model (via the flow equations in Theorem \ref{flowformula} below).

We note that in \cite{ryz2} the key to the stabilization result lies in a compactness criterion (given in this treatment as Lemma \ref{compactnesscriterion}) for flow trajectories wherein the flow is bounded in higher norms by the plate trajectory, also in higher norms. These estimates are obtained via the thermoelastic character of the structural dynamics. As no such approach applies here, we take a different tact: we appeal to a more classical approach \cite{ball} wherein we first show that the dynamics strongly stabilize to a stationary point when {\em regular} initial data is considered; this requires giving an a priori bound {\it uniform-in-time}  on regular solutions in higher norms. To do so, we must demonstrate propagation of regularity first on the finite time horizon for the full flow-plate dynamics, followed by operating on the delayed plate dynamics in order to show additional regularity (for regular data) on the infinite time horizon which requires large viscous damping. With this regularity in hand, we show that strong solutions stabilize (strongly) to an equilibria in the sense of Theorem \ref{conequil1}. We then apply an approximation argument to pass this convergence property onto finite energy solutions.

\section{Supportive results}
We begin with two preliminary results which (together) will  provide a skeleton for the proofs of the  main results  formulated in the previous section.
These results, while supporting the final conclusions, are each of interest on their own.

The first result---formulated  already in Theorem \ref{regresult}---shows that for {\em regular initial data} the flow-plate dynamics converge strongly to equilibria. For this result, no assumptions on the damping coefficients are necessary, simply $k>0$ and $\beta \ge 0$.

The second preliminary result provides uniform-in-time Hadamard continuity for the  semigroup $S_t(\cdot)$ from Theorem \ref{nonlinearsolution} on the {\em infinite time horizon}. For this result the sufficiently large minimal damping parameters $k_0$ and $\beta_0$ are necessary. In general, the sensitivity analysis for the flow-plate dynamics is very subtle on the infinite time horizon. Here, we exploit the fact that sufficiently large plate damping gives  control of convergence rate  of the difference of two (delay) plate trajectories (Theorem \ref{exp} below).
\begin{theorem}\label{strongcont}
With reference to  solutions described by Theorem \ref{nonlinearsolution},  and assuming that the damping parameters are sufficiently large $ k \geq  k_0 > 0$, $\beta \ge \beta_0>0$,  the semigroup $S_t(\cdot)$ is uniform-in-time Hadamard continuous, i.e., for any sequence $y_m^0 \to y^0$ in $Y$ and any $\epsilon>0$ there is an $M$ so that for $m>M$
$$\sup_{t>0} \|S_t(y^0_m) - S_t(y^0)\|_{Y_{\rho}} < \epsilon. $$
\end{theorem}

\begin{remark}
From this point on, when we refer to convergence of full flow-plate trajectories, we consider the state space topologized by the local flow energy convergence, i.e. we consider the norm of the space $Y_{\rho} = Y_{pl} \times  Y_{fl,\rho}$ with any $\rho>0$ fixed. This will yield the final convergence in $\widetilde Y$.\end{remark}

\subsection{Weak to strong---Proof of main theorem \ref{conequil1} from the Theorems \ref{regresult} and \ref{strongcont}}\label{7}
 Given the validity of the two supportive theorems   Theorem \ref{regresult} and Theorem \ref{strongcont} (to be proven below) the  proof of Theorem \ref{conequil1}  is straightforward. We must improve the convergence to the equilibria set  for smooth initial data (Theorem \ref{regresult}) to finite energy initial data residing  in $Y$.  This is done via the  uniform-in-time estimate in Theorem \ref{strongcont}. With slight abuse of notation (for ease of exposition), let $ y^0 = ( u_0,  u_1,  \phi_0,  \phi_1 ) \in Y$ (the initial data) and let  $y^0_m  = ( u^0_{m},  u^1_{m} ,  \phi^0_{m}, \phi^1_{m}  ) \in \cD(\bT) \cap B_R(Y)$ be chosen such that
$ y^0_{m}\rightarrow y^0 $ strongly in $Y$. The above choice is possible by the virtue of the density of $\cD(\bT) $ in $Y$ \cite{webster}.

Let $\epsilon>0$ be given.  We note that by Theorem \ref{strongcont},  we may choose an $M$ sufficiently large so that for all $k>k_0(R)$ and $\beta > \beta_0(R)$ and for all $m \ge M$ $$ \sup_{t > 0} \|S_{t}(y^0 ) -  S_{t} (y^0_{m}) \|_{Y_{\rho}}< \epsilon/2.$$
 Fix $M$. Theorem \ref{regresult} guarantees that 
\begin{equation}\label{start1}
\lim_{t \rightarrow \infty} d_{Y_{\rho}}(S_{t}(y_M^0),\mathcal N)=0.
\end{equation}
Hence, there exists a $T(\epsilon,M)$ so that for all $t>T(\epsilon,M)$
\begin{align}\label{start}
d_{Y_{\rho}}(S_t(y_M^0),\mathcal N) < \dfrac{\epsilon}{2}.
\end{align}
Then
\begin{align}
d_{Y_{\rho}}(S_t(y^0),\mathcal N) \le &~ d_{Y_{\rho}}(S_t(y^0),S_t(y_M^0))+d_{Y_{\rho}}(S_t(y_M^0),S_t(y^0)) \\
< &~\epsilon
\end{align}
for all $t>T$. In the last line we have critically used both of Theorem \ref{regresult} and \ref{strongcont}. This concludes the proof of Theorem \ref{conequil1} from the fundamental theorems.

The bulk of the treatment is thus devoted to the proofs of supporting theorems, and the proof of weak convergence of the full dynamics (Theorem \ref{weakth}).

\section{Preliminaries}
In the following sections we discuss key results from \cite{delay} which provide the existence of an attracting set for the plate dynamics, and also key results utilized in \cite{b-c,b-c-1,springer,ryz,ryz2} concerning a decomposition of the flow component of the dynamics.

\subsection{Decomposition of the flow}
We first decompose the flow problem from \eqref{flow} into two pieces corresponding to zero Neumann data, and zero initial data, respectively:
\begin{equation}\label{flow1}
\begin{cases}
(\partial_t+U\partial_x)^2\phi^*=\Delta \phi^*  & ~\text{ in }~ \realsthree_+ \times (0,T) \\
\phi^*(0)=\phi_0; ~~\phi_t(0)=\phi_1 \\
\Dn \phi^* = 0 & ~\text { in } ~\partial \realsthree_+ \times (0,T)
\end{cases}
\end{equation}
\begin{equation}\label{flow2}
\begin{cases}
(\partial_t+U\partial_x)^2\phi^{**}=\Delta \phi^{**}  & ~\text{ in }~ \realsthree_+ \times (0,T) \\
\phi^{**}(0)=0; ~~\phi_t^{**}(0)=0 \\
\Dn \phi^{**} = h(\xb,t) & ~\text { in } ~\partial \realsthree_+ \times (0,T)
\end{cases}
\end{equation}
where
\begin{equation}\label{h}
h(\xb, t) \equiv u_t + U u_x
\end{equation}
Following the analyses in \cite{ryz,ryz2,springer,b-c,b-c-1} we denote the solution to \eqref{flow1} as $\phi^*$ and the solution to \eqref{flow2} as $\phi^{**}$; then, the full flow solution $\phi$ coming from \eqref{flowplate} has the form $$\phi(t)=\phi^*(t)+\phi^{**}(t)$$ where $\phi^* (t) $ solves (\ref{flow1}) and $\phi^{**} (t) $ solves (\ref{flow2}).
\begin{remark}
The analysis of $\phi^*$ is identical to that given in \cite{springer,ryz,b-c,b-c-1}. However, the treatment of the second part
$\phi^{**} (t)$, which corresponds to the  hyperbolic Neumann map,  is very different due to the loss of derivative in Neumann map. Indeed, with rotational inertia in place one has for finite energy solutions  $h \in C(0, T; H^1(\Omega) ) $.
On the other hand from  \cite{miyatake1973} $ h \in L_2(H^{1/2}(\realstwo) ) \mapsto
\phi^{**} \in  C(H^1(\realsthree_+) \cap C^1(L_2(\realsthree_+)) $---where the latter  is of  finite energy. Thus the Neumann map is compact in this case.
In the absence of rotational inertia one only has  $h \in C(L_2(\realstwo)) $. The above regularity does not produce finite energy
solutions,  with a maximal regularity being $ \phi^{**} \in  C(H^{2/3}(\realsthree_+)) \cap C^1(H^{-1/3}(\realsthree_+)) $ yielding  the loss of $1/3 $ derivative. This loss is optimal and cannot be improved  \cite{tataru}. This fact   clearly underscores that the component-wise analysis of finite energy solutions  successfully performed in the  past literature
\cite{springer,ryz,ryz2} cannot be contemplated.
  \end{remark}
 As mentioned before, for the analysis of $\phi^*$, we use the tools developed in \cite{b-c-1,b-c}.
Using the Kirchhoff type representation for the solution  $\phi^*(\xb,t)$
in $\R_+^3$ (see, e.g., Theorem~6.6.12 in \cite{springer}), we can conclude that
if the initial data   $\phi_0$ and $\phi_1$ are  localized in the ball $ K_{\rho}$,
then by  finite dependence on the domain of the signal in three dimensions  (Huygen's principle),
   one obtains  $\phi^*(\xb,t)\equiv 0$ for all $\xb\in  K_{\rho} $
and $t\ge t_{\rho}$. Thus we have that
\[
\big(\partial_t+U\partial_x\big)tr[\phi^*]\equiv0,~~~\xb\in \Omega,~t\ge t_{\rho}.
\]
Thus $\phi^*$ tends to zero in the sense of the local flow energy, i.e.,  \begin{equation}\label{starstable} \|\nabla \phi^*(t)\|_{L_2( K_{\rho} )}^2 + \|\phi^*_t(t)\|_{L_2( K_{\rho} )} \to 0, ~~ t \to \infty,\end{equation} for any fixed $ \rho>0$.

We now introduce a compactness criterion for the local energy convergence above:
\begin{lemma}\label{compactnesscriterion}
Let $\{(\phi_0^m,\phi_1^m)\}_m^{\infty}$ be a bounded sequence in $Y_{fl}=H^1(\realsthree_+) \times L_2(\realsthree_+)$ and let $\eta>0$. If for any $ \rho>0$ there exists an $N(\rho)$ and $C(\rho)$ so that
$$\|\nabla \phi_0^m\|^2_{\eta, K_{\rho} }+\|\phi_1^m\|^2_{\eta, K_{\rho} } \le C(\rho)$$ for all $m>N(\rho)$ then the sequence $\{(\phi_0^m,\phi_1^m)\}_m^{\infty}$ is compact in $\widetilde Y_{fl}$.
\end{lemma}
\noindent This is given as Lemma 10 in \cite[p. 472]{ryz} (where it is proved) and is utilized in a critical way in \cite{ryz2} as well.

For the term $\phi^{**}$ we have the following theorem which provides us with an explicit form of the solution (for a proof, see for instance \cite[Theorem 6.6.10]{springer}).
Below, we utilize the notation: $$f^{\dag}(\xb,t,s,\theta)=f\left(x-\kappa_1(\theta,s,z), y-\kappa_2(\theta,s,z), t-s\right),$$ and $$\kappa_1(\theta,s,z)=Us+\sqrt{s^2-z^2}\sin \theta,~~\kappa_2(\theta,s,z) = \sqrt{s^2-z^2}\cos \theta.$$
\begin{theorem}\label{flowformula}
Considering the problem in \eqref{flow2} with  zero initial flow data, and considering $h(x,y,t) =-(u_t (x,y,t)+Uu_x(x,y,t))\mathbf{1}_{\Omega}$, there exists a time $t^*(\Omega,U)$ such that we have the following representation for the weak solution $\phi^{**}(t)$ for $t>t^{*}$:
\begin{equation}
\phi^{**}(\xb,t) = -\dfrac{\chi(t-z) }{2\pi}\int_z^{t^*}\int_0^{2\pi}(u^{\dag}_t(\xb,t,s,\theta)+Uu^{\dag}_x(\xb,t,s,\theta))d\theta ds.
\end{equation}
where $\chi(s) $ is the Heaviside function. The time $t^*$ is given by:
\begin{equation*}
   t^*=\inf \{ t~:~\xb(U,\theta, s) \notin \Omega \text{ for all } (x,y) \in \Omega, ~\theta \in [0,2\pi], \text{ and } s>t\}
\end{equation*} with $\xb(U,\theta,s) = (x-(U+\sin \theta)s,y-s\cos\theta) \subset \realstwo$ (not to be confused with $\xb$ having no dependencies noted, which is simply $\xb = (x,y)$).

Moreover, we have the following point-wise formula for the derivatives in $t$ and $\xb$ \cite[p. 480]{ryz2}(which are justified for plate solutions with initial data in $\cD(\bT)$, and can be taken distributionally for data in $Y$):
\begin{align}\label{formderiv1}
\phi^{**}_t&(\xb, t) =  \dfrac{1}{2\pi}\Big\{\int_0^{2\pi}u_t^{\dag}(\xb,t,t^*,\theta) d\theta -\int_0^{2\pi} u_t^{\dag}(\xb,t,z,\theta )d\theta \\\nonumber
&+U\int_z^{t^*}\int_0^{2\pi}[\partial_x u_t^{\dag}](\xb,t,s,\theta) d\theta ds+\int_z^{t^*}\int_0^{2\pi}\dfrac{s}{\sqrt{s^2-z^2}}[M_{\theta}u_t^{\dag}](\xb,t,s,\theta) d\theta ds\Big\} 
\end{align}
with $M_{\theta} = \sin\theta\partial_x+\cos \theta \partial_y$.
\end{theorem}

We note that  with  $(\phi_0,\phi_1) \in H^1(\R_+^3)\times  L_2(\R_+^3) $
one obtains \cite{miyatake1973,supersonic}
 $(\phi^*(t),  \phi^{*}_t(t)) \in   H^1(\R_+^3)\times  L_2(\R_+^3)$. Thus, by Theorem \ref{nonlinearsolution}
 we also have that
 $$
 (\phi^{**}(t),  \phi^{**}_t(t)) \in  H^1(\R_+^3)\times  L_2(\R_+^3).
 $$
 \begin{remark} Note that \textit{this last property is not valid for a flow solution with $L_2$ boundary Neumann data.} The general theory will provide at most $ H^{2/3}(\R_+^3\times [0,T])$ \cite{tataru}.  However, the improved regularity is due to the interaction with the plate and the resulting cancellations on  the interface. Moreover, we also obtain a meaningful ``hidden trace regularity" for the aeroelastic potential on the boundary of the structure \cite{supersonic}:
  \begin{equation}\label{trace}
  (  \partial _{t} + U \partial_{x} )\gamma [ \phi ] \in L_2(0, T; H^{-1/2}(\Omega) )
  \end{equation}
  where $T$ is arbitrary. \end{remark}

Additionally we have the following bounds on solutions in higher norms. These were critical in the previous analyses of this problem which made use of inertial terms $-\alpha \Delta u_{tt}$ and strong damping $-k_2\Delta u_t$ \cite{b-c,b-c-1} or thermal smoothing \cite{ryz,ryz2}. In particular the following inequality has been shown in \cite{ryz}  labeled Lemma 8 (p. 469) and (56) (p. 479):
\begin{lemma}\label{compact2}
For \eqref{flow2} taken with $h(\xb,t)=-(u_t+Uu_x)\mathbf{1}_{\Omega}$, we have
\begin{align}
\|\nabla  \phi^{**}(t)\|^2_{\beta, K_{\rho} }&+\|\phi_t^{**}(t)\|^2_{\beta, K_{\rho} }\nonumber \\
\le  ~C(\rho)& \big\{\|u(t)\|^2_{H^{s+\beta}(t-t^*,t;H_0^{2+\beta}(\Omega))}+\|u_t(t)\|^2_{H^{s+\beta}(t-t^*,t;H_0^{1+\beta}(\Omega))} \big\}
\end{align} for $0<s+\beta<1/2$ and $t>t^*(U,\Omega)$.
\end{lemma}

\subsection{Global Attracting Sets  for Plate Dynamics}\label{attractorresult}

\begin{theorem}\label{rewrite}
Let the hypotheses of Theorem~\ref{nonlinearsolution} be in force,
and $(u_0,u_1;\phi_0,\phi_1) \in H_0^2(\Omega) \times L_2(\Omega) \times H^1(\realsthree_+) \times L_2(\realsthree_+)$.
Assume that there exists an $\rho_0$ such that $\phi_0(\xb) = \phi_1(\xb)=0$ for $|\xb|>\rho_0$.
Then the there exists a time $t^{\#}(\rho_0,U,\Omega) > 0$ such that for all $t>t^{\#}$ the weak solution $u(t)$ to
(\ref{flowplate})  satisfies the following equation:
\begin{equation}\label{reducedplate}
u_{tt}+\Delta^2u+ku_t+\beta u-[u,v(u)+F_0]=p_0-(\partial_t+U\partial_x)u-q^u(t)
\end{equation}
with
\begin{equation}\label{potential}
q^u(t)=\dfrac{1}{2\pi}\int_0^{t^*}ds\int_0^{2\pi}d\theta [M^2_{\theta}\widehat u](x-(U+\sin \theta)s,y-s\cos \theta, t-s).
\end{equation}
Here, $\widehat u$ is the extension
 of $u$
   by 0 outside of $\Omega$. $M_{\theta} $ is a first order differential operator given by:
$M_{\theta} \equiv \sin \theta \partial_x + \cos \theta \partial_y $
\end{theorem}
\noindent The proof of this theorem depends on the decomposition of the flow given in the previous section. See \cite{b-c,b-c-1, springer} for details.
\begin{remark}
This extremely helpful theorem first appeared as a heuristic in \cite{kras} and was used in this way for many years; it was later made rigorous in \cite{Chu92b} and used extensively in \cite{springer,delay}.
\end{remark}

 Thus, after some time, the behavior of the flow can be captured by the aerodynamical pressure term $p(t)$ in the form of a reduced delayed potential.  Theorem~\ref{rewrite} allows us (assuming that the flow data is compactly supported) to suppress the dependence of the problem on the flow variable $\phi$.
Here we emphasize that the structure of aerodynamical pressure \eqref{aero-dyn-pr}  leads to the velocity term  $-u_t$ on the RHS of \eqref{reducedplate}.

We may utilize this as natural damping appearing in the structure of the reduced flow pressure by moving this term to the LHS. This was the topic of the treatment \cite{delay}, wherein this damping---acquired from the flow---was sufficient to obtain convergence to a compact set {\em for the (delayed) plate dynamics}.  While such result has been known for plates with rotational terms \cite{springer}, its validity for non-rotational model  has been established only recently \cite{delay}.
 Precise formulation of this result is  given below as a theorem:

 \begin{theorem}\label{th:main2}
Suppose $0\le U \ne 1$,  $F_0 \in H^{4}(\Omega)$ and $p_0 \in L_2(\Omega)$. We may also take $k \ge 0$, $\beta \ge 0$ on $\Omega$.
 Then there exists a compact set $\mathscr{U} \subset H_0^2(\Omega) \times L_2(\Omega)$ of finite fractal dimension such that $$\lim_{t\to\infty} d_{Y_{pl}} \Big( (u(t),u_t(t)),\mathscr U\Big)=\lim_{t \to \infty}\inf_{(\nu_0,\nu_1) \in \mathscr U} \big( \|u(t)-\nu_0\|_2^2+\|u_t(t)-\nu_1\|^2\big)=0$$
for any weak solution $(u,u_t;\phi,\phi_t)$ to (\ref{flowplate}), and with
initial data
$$
(u_0, u_1;\phi_0,\phi_1) \in  H_0^2(\Omega)\times L_2(\Omega)\times H^1(\realsthree_+)\times L_2(\realsthree_+)
$$
which are
localized  in $\R_+^3$ (i.e., $\phi_0(\xb)=\phi_1(\xb)=0$ for $|\xb|>\rho_0$ for some $ \rho_0>0$). We have the extra regularity $\mathscr{U} \subset \big(H^4(\Omega)\cap H_0^2(\Omega) \big) \times H^2(\Omega)$, and any plate trajectory $(u,u_t)$ on the plate attracting set (namely with initial $(u_0,u_1) \in \mathscr U$) has the additional property that $$(u,u_{t},u_{tt}) \in C\left((0,\infty);(H^4\cap H_0^2)(\Omega) \times H_0^2(\Omega) \times L_2(\Omega)\right).$$
\end{theorem}
\begin{remark}
We emphasize that, above, we do not require {\em any} damping imposed in the model. The flow itself provides stabilization to a compact set.
\end{remark}
A further Lemma, which follows from the proof of Theorem \ref{th:main2} \cite[Lemma 4.8]{delay} will be needed below. Specifically, an absorbing ball is constructed via a Lyapunov approach in Section 4.3 of \cite{delay}---this ultimately gives {\em dissipativity} of the delay dynamical system. Here, we care that this calculation is unaffected by the size of the damping parameters.
\begin{lemma}\label{indeppara}
Suppose $0\le U \ne 1$,  $F_0 \in H^{4}(\Omega)$ and $p_0 \in L_2(\Omega)$. Consider initial data
$$
(u_0, u_1;\phi_0,\phi_1) \in  H_0^2(\Omega)\times L_2(\Omega)\times H^1(\realsthree_+)\times L_2(\realsthree_+)
$$
which are
localized  in $\R_+^3$ (i.e., $\phi_0(\xb)=\phi_1(\xb)=0$ for $|\xb|>\rho_0$ for some $ \rho_0>0$). Then there is absorbing set $\mathscr B \subset H_0^2(\Omega)\times L_2(\Omega)$ which is uniform with respect to the initial data $y_0 $ (it may depend on $\rho_0$ though)  and uniform in the damping parameters $k,\beta \ge 0$. \end{lemma}
\begin{remark}
The proof of Lemma \ref{indeppara} follows from a careful analysis of the Lyapunov function constructed in \cite{delay}; indeed, the damping parameters affect the {\em time of absorption} $T(\mathscr B)$ for a given plate trajectory $(u,u_t)$. The details related to the influence of the damping parameters are the same as in the proof of Theorem 3.10 in \cite{Memoires} [note that Assumptions  3.23. 3.24 in \cite{Memoires} are satisfied). See also Theorem 4.3 in \cite{Memoires}. (The effect of $k$ and $\beta$ on the Lyapunov calculations can be seen in the proof of Theorem \ref{exp}, though the structure of the Lyapunov functions is different and arises in a different context.)
\end{remark}

In \cite{delay} the interaction between the full flow-plate dynamics generated by \eqref{flowplate}, encapsulated by the semigroup $S_t(\cdot)$ on $Y$, and the dynamics generated by the delayed plate equation \eqref{reducedplate} above are discussed in detail. After a sufficiently large time $t^{\#}(U,\rho_0,\Omega)$, solutions to \eqref{flowplate} must also satisfy \eqref{reducedplate}, and we utilize energy methods and dynamical systems techniques on the semiflow generated by the reduced plate dynamics. The dynamics given by \eqref{reducedplate} are encapsulated by a semiflow $T_t$ on the space $X= H_0^2(\Omega) \times L_2(\Omega) \times L_2(-t^*,0;H_0^2(\Omega))$, where $t^*$ is the time of delay as given in Theorem \ref{flowformula}. For $(u_0,u_1,\eta)=x_0 \in X$, the dynamics are given by $T_t(x_0)=(u(t),u_t(t),u^t)$ where $u^t=u(t+s)$ for $s \in (-t^*,0)$; the initial data for the delay component is $\eta=u\big|_{t \in (-t^*,0)}$.

We will need an additional estimate on the plate trajectories which further improves the structure of the long-time behavior estimates associated to the analysis above. In \cite[Section 4.5]{delay} we show a {\em quasistability} estimate on the delay plate trajectories which lie on the attractor. Such an estimate is the key in the analysis in \cite[Section 4.5]{delay} and directly leads to the finite dimensionality and additional smoothness of $\mathscr U$ in Theorem \ref{th:main2}. The only properties of the attracting set utilized in obtaining this estimate are its {\em compactness} and {\em invariance under the dynamics}---both of which are available for $k=\beta=0$. By taking $k$ sufficiently large we can show the quasistability estimate on any forward invariant set---rather than on the attractor. (Such estimate valid for an arbitrary invariant bounded  will lead to a construction of exponential attracting set---\cite{Miranville}). However, for the sake of the sensitivity arguments we utilize below, we need a stronger, more general estimate which provides uniform {\em exponential decay} on the difference of two solutions to the delay plate equation given in Theorem \ref{rewrite}.  This requires additionally taking $\beta$ sufficiently large.

\begin{theorem}\label{exp}
Let $y_i(t)=(u_i(t),u_{it}(t);\phi_i(t),\phi_{it}(t))$ for $i=1,2$ be two solutions to \eqref{flowplate} such that $$ \|(u_i(t),u_{it}(t))\|^2_{H_0^2(\Omega)\times L_2(\Omega)} \le R^2$$ for all $t>0$; we assume that the flow components of the trajectories have data $\phi_i(t_0)$ and $\phi_{it}(t_0)$ which vanish outside of some $K_{\rho_0}$ (so Theorem \ref{rewrite} applies). Write $\bu(t)=u_1(t)-u_2(t)$. Then for some $t_0>t^{\#}(\rho_0,U,\Omega)$ there exists $k_0 (R)  >0, \beta_0 (R)  >0 $ and a constant
  $\mathcal C(k_0,\beta_0)>0$ such that we have the (delay-type) estimate:
\begin{align}
 \|\Delta \bu(t)\|^2 + \|\bu_t(t)\|^2
\le &~e^{-\mathcal C(t-t_0) }\big[\|(\bu(t_0),\bu_t(t_0),\bu^{t_0})\|^2_X \big]
\notag\\
 \le &~ C(R)e^{-\mathcal C t}
\end{align}
\end{theorem}
\begin{corollary}
Considering Lemma \ref{indeppara}, and Theorem \ref{exp} above, we see that for $t_0$ sufficiently large the exponential decay can be made to depend only on the size of the absorbing ball $\mathscr B$, which in turn depends on the intrinsic parameters of the problem---$U,\Omega,\rho_0,p_0,$ and $F_0$.
\end{corollary}
\begin{remark}
Since $k_0, \beta_0 $ depend on $R$,  it is essential that $R$ can be made independent on the large values of damping $k, \beta$. This is possible due to the fact that the absorbing ball for the delayed plate equation can be made independent on damping parameters $ k >0, \beta >0$. For arguments involving $t\to +\infty$ we may consider a time $t_0$ sufficiently large (depending on the initial plate and flow data) such that the trajectory $(u,u_t)$ enters the absorbing set $\mathscr B$. At which point we may utilize the result of Theorem \ref{exp}. This is ultimately how we will see that our asymptotic-in-time result---which depends on scaling the minimal damping parameters $k_0$ and $\beta_0$ sufficiently large---depends only on the size of the absorbing ball for the plate dynamics.
\end{remark}

  This is an estimate for the exponential continuity property on differences (resulting directly from the presence of the damping)---we are careful to point out that this property is derived from the delay plate dynamics, and hence the result depends critically on the compact support of the flow data by introducing the delay component, but not on the particular initial data, nor size of the initial data, in $Y$.

The proof of this property  utilizes a  Lyapunov technique adapted from the proof of dissipativity of the delay dynamical systems in \cite{springer,delay}. We give this argument below.
\begin{remark}
This is the key point in the treatment where we must critically utilize the sufficiently large damping coefficient $k_0$ and $\beta_0$.
\end{remark}
\begin{proof}[Proof of Theorem \ref{exp}]
 We make use of the following notation:
\begin{align}
E_{\bu}(t) =& ~\dfrac{1}{2}\left\{\|\Delta \bu(t)\|^2+\|\bu_t(t)\|^2+\beta\|\bu(t)\|^2\right\};\\
\cF =& ~f_V(u_1)-f_V(u_2);~~~
q^{\bu} = q^{u_1}-q^{u_2}.
\end{align}
Our Lyapunov function will depend on the difference of two solutions (i.e. the variable $\bu$) and we consider time sufficiently large so we may refer to the delay plate equation \eqref{reducedplate}:
\begin{equation}
V(\bu(t)) \equiv E_{\bu}(t) +\nu\Big\{\langle \bu_t,\bu\rangle +\frac{k}{2}\|\bu\|^2\Big\}+\mu\int_0^{t^*}\int_{t-s}^t\|\Delta \bu(\tau)\|^2 d\tau ds
\end{equation}

\noindent  We now proceed to show that $$\dfrac{d}{dt}V(\bu(t)) \le - C(\mu,\nu,\beta_0,k_0)V(\bu(t)).$$
Returning to the delay equation \eqref{reducedplate} we have:
\begin{align}
\dfrac{d}{dt}V(\bu(t)) = &~\dfrac{d}{dt}E_{\bu}(t) +\nu\Big\{\langle \bu_{tt},\bu\rangle +\|\bu_t\|^2+k\langle \bu_t,\bu\rangle \Big\}\\&+\mu t^*\|\Delta \bu(t)\|^2-\mu\int_{t-t^*}^t\|\Delta \bu\|^2 \nonumber  \end{align}
Using:
\begin{align*}
\langle \bu_{tt}+\beta \bu, \bu_t \rangle =&~-\langle \Delta \bu ,\Delta \bu_t \rangle -(k+1)\|\bu_t\|^2+\langle \cF , \bu_t \rangle-U\langle \bu_x, \bu_t \rangle - \langle q^{\bu},\bu_t \rangle \\
\langle \bu_{tt}+k\bu_t,\bu\rangle =& -\|\Delta \bu \|^2 + \langle \cF,\bu \rangle-\beta \|\bu\|^2-\langle \bu_t , \bu \rangle -U\langle \bu_x, \bu \rangle-\langle q^{\bu}, \bu \rangle,
\end{align*}
We have \begin{align*}
\dfrac{d}{dt}E_{\bu}(t) = & \langle \bu_{tt},\bu_t \rangle+\langle \Delta \bu,\Delta \bu_t \rangle+\beta \langle \bu,\bu_t \rangle\\
 = &~-(k+1)\|\bu_t\|^2+\langle \cF , \bu_t \rangle-U\langle \bu_x, \bu_t \rangle - \langle q^{\bu},\bu_t \rangle
\end{align*}
Hence \begin{align}
\dfrac{d}{dt}V(\bu(t)) =&~ -(k+1)\|\bu_t\|^2-\nu\|\Delta \bu\|^2-\beta\nu \|\bu\|^2 -\mu\int_{t-t^*}^t\|\Delta \bu\|^2\\ \nonumber
&+ \langle \cF,\bu_t \rangle + \nu \langle \cF,\bu\rangle-U\langle \bu_x, \bu_t\rangle-U\nu \langle \bu_x, \bu \rangle  \nonumber \\\nonumber
&+ \nu \|\bu_t\|^2 +\mu t^* \|\Delta \bu\|^2-\langle q^{\bu}, \bu_t \rangle - \nu \langle q^{\bu}, \bu \rangle-\nu\langle \bu_t, \bu \rangle
\end{align}
We must control the nonnegative terms above; we utilize Young's inequality, compactness of the Sobolev embeddings, and the locally Lipschitz character $f_V$ \cite[p.44]{springer} $$\|\cF\|_{-\delta} \le C(R)\|\bu\|_{2-\delta},~\delta \in [0,1/2).$$ This yields:
\begin{align*}
\langle\cF, \bu_t \rangle \le&~ \epsilon \|\Delta \bu\|^2+\dfrac{C(R)}{\epsilon}\|\bu_t\|^2;\hskip.5cm
\nu \langle \cF,\bu\rangle \le  ~\epsilon \|\Delta \bu\|^2+\dfrac{\nu^2 C(R)}{\epsilon}\|\bu\|^2\\
U\langle \bu_x, \bu_t\rangle \le &~\epsilon \|\Delta \bu\|^2+\dfrac{C}{\epsilon}\|\bu_t\|^2;\hskip.5cm
U\nu \langle \bu_x, \bu \rangle \le ~\epsilon \|\Delta \bu\|^2+\dfrac{C \nu^2}{\epsilon}\|\bu\|^2\\
\nu\langle \bu_t, \bu \rangle \le &~ \epsilon \|\bu\|^2+\dfrac{C\nu^2}{\epsilon}\|\bu_t\|^2;\hskip.5cm
\langle q^{\bu}, \bu_t \rangle \le ~\delta \|q^{\bu}\|^2+\dfrac{C}{\delta}\|\bu_t\|^2\\
\nu \langle q^{\bu}, \bu \rangle \le &~\delta \|q^{\bu}\|^2+\dfrac{C\nu^2}{\delta}\|\bu\|^2
\end{align*}
for all $\epsilon,\delta>0$; the constants $C$ above are ubiquitous and may vary from line to line (we do not combine them).

We have the direct estimate:
$$\|q^{\bu}(t)\|^2 \le C \int_{t-t^*}^t\|\Delta \bu(\tau)\|^2 d\tau.$$
Utilizing this and combining:
\begin{align}
\dfrac{d}{dt}V(\bu(t)) \le &~\left\{\dfrac{C(R)}{\epsilon}+\dfrac{C}{\epsilon}+\dfrac{C\nu^2}{\epsilon}+\dfrac{C}{\delta}+\nu-(k+1)\right\}\|\bu_t\|^2\label{e1}\\
&+\left\{\dfrac{\nu^2 C(R)}{\epsilon}+\dfrac{C \nu^2}{\delta}+\epsilon -\beta\nu\right\} \|\bu\|^2\label{e2}\\
&+\left\{4\epsilon+\mu t^*-\nu\right\}\|\Delta \bu\|^2\label{e3}\\
&+ \left\{\delta C-\mu\right\}\int_{t-t^*}^t\|\Delta \bu\|^2\label{e4}
\end{align}
We  first choose
 $\nu$ (small) so that
 \begin{equation}\label{VE} E_{\bu}(t) \le C\big(\nu \big)V(\bu(t)).\end{equation}

We then select  $\epsilon(\nu)$ and $~\mu(\nu)$; then we choose $~ \delta(\mu)$; and finally choose $~ \beta(\nu,\epsilon,\delta)$. These selections will guarantee that the terms in \eqref{e2}--\eqref{e4} (each depending on $\nu$) are negative. Thus, the underlying parameters $t^*,R$ and $\Omega$ provide the minimal damping coefficient $\beta_0 \le \beta$.
 Lastly, we choose $k\ge k_0$ sufficiently large so as to guarantee
\begin{align}
\dfrac{d}{dt}V(\bu(t)) \le -C(k)\Big\{\|\bu_t\|^2+\|\Delta \bu\|^2+\beta\|\bu\|^2+\int_{t-t^*}^t\|\Delta \bu\|^2\Big\}.
\end{align}
But since $$t^* \int_{t-t^*}^t\|\Delta \bu\|^2 d\tau \ge  \int_0^{t^*}\int_{t-s}^t\|\Delta \bu(\tau)\|^2 d\tau ds,$$
we have
$$\dfrac{d}{dt}V(\bu(t)) \le -\mathcal C(k,\beta)V(\bu(t)).$$
This gives exponential decay of $V(\bu(t))$.
 This implies the exponential decay of $E_{\bu}(t)$:
\begin{align} E_{\bu}(t) \le&~ e^{-\mathcal C (t-t_0)}V(\bu(t_0))\notag\\ \le&~ e^{-\mathcal C (t-t_0)}[E_{\bu}(t_0)+\int_{t_0-t^*}^{t_0}\|\Delta \bu \|^2] \notag\\ \le &~ C(R)e^{-\mathcal C t}.\end{align}
\end{proof}

In view of the estimate asserted by Theorem \ref{exp}, it is evident that the main issue to contend with is showing strong convergence of flow solutions.
While dispersive effects resolve the issue of the convergence for the first component of the flow $\phi^*$, the lack of  boundedness (and of course of compactness)  for the Neumann map presents the main predicament.
The analysis of this issue comprises most of the technical  portion  of the paper.

\section{Outline of the remainder of the paper}
The proof of the main result Theorem \ref{conequil1} will be broken up into steps, based on the preliminary Theorems \ref{regresult} and \ref{strongcont}. These theorems will be themselves proved by a sequence of smaller lemmas, described below.
\begin{enumerate}
\item ({\em Proof of Theorem \ref{regresult}}---Section \ref{5}) First, we will show that, independent of any results on the strong convergence of the flow, the plate velocities go strongly to 0 in $L_2(\Omega)$; this is given in Theorem \ref{convergenceprops}. We note, at this point, that the existence of the compact attracting set above in Theorem \ref{th:main2} gives that there exists some limit point for $u(t_n)$ in $H_0^2(\Omega)$ for any sequence $t_n \to \infty$. WLOG we restrict to a sequence $t_n$ and hence the boundedness of the semiflow $S_t$ on $\widetilde Y$ yields a weak limit point, i.e., we have that $S_{t_n}(y_0)=(u(t_n),u_t(t_n);\phi(t_n),\phi_t(t_n)) \rightharpoonup (\hat u, 0, \hat \phi, \hat \psi)$ in $Y_{Y_{\rho}}$. This motivates Theorem \ref{weakth}.

 We then operate on regular trajectories from $\cD(\bT)$. On the regular trajectories we will show strong stability (in the sense of Theorem \ref{conequil1}) of the problem for regular initial data without ``largeness" assumptions on the damping coefficent. The key point in this step is demonstrating global-in-time boundedness of plate trajectories in higher norms (Theorem \ref{highernorms}). This requires us to operate on full-flow plate trajectories {\em first} in order to obtain propagation of regularity on the finite horizon; at a given $t^{\#}$ the reduction result (Theorem \ref{rewrite}) is employed in order to operate on delayed plate trajectories. At this point, the requisite ultimate boundedness is obtained. Using the techniques from \cite{ryz,ryz2} (Lemma \ref{compactnesscriterion} and Lemma \ref{compact2}) and using Theorem \ref{flowformula} we will identify the a weak limit point in Lemma \ref{limitpoint}. Finally, we will identify this limit point as being a stationary point ($\phi_t(t_n) \to 0$) in Lemma \ref{characterize}.

At this point we utilize a limiting procedure, along with our results up to this point, to show that our subsequential limit from the previous steps satisfies the stationary flow-plate problem \eqref{static}. This will be given as Lemma \ref{staticsols}.
\item ({\em Proof of Theorem \ref{strongcont}}---Section \ref{6}) We show a uniform-in-time continuity property (under a large damping assumption) of the dynamics which enable  an approximation argument (as in \cite{koch}) to pass convergence obtained for regular initial data onto energy type solutions. Specifically we consider a sequence $y^0_m \to y^0$ in the state space $Y$, with $y_m \in \cD(\bT)$ (the domain of the semigroup generator). We show that the semigroup $S_t(\cdot)$ uniformly-in-time Hadamard continuous. We then consider the difference of $S_t(y_M)$ and $S_t(y^0)$ in $Y_{\rho}$ (for some $M$ sufficiently large) in order to pass the limiting behavior {\em in time} of $S_t(y^0_M)$ onto $S_t(y^0)$.
\end{enumerate}
\section{Proof of  Theorem \ref{regresult}}\label{5}

\subsection{Plate Convergence}
We begin by noting that the existence of the compact attracting set for the plate dynamics in Theorem \ref{th:main2}  we infer that for any initial data $y^0=(u^0,u^1,\phi^0,\phi^1)\in Y$ and any sequence of times $t_k \to \infty$ there exists a subsequence of times $t_{n_k} \to \infty$ and a point $(\hat u, \hat w) \in Y_{pl}=H_0^2(\Omega)\times L_2(\Omega)$ such that $(u(t_{n_k}),u_t(t_{n_k)}) \to (\hat u, \hat w)$ strongly in $Y_{pl}$. (Here and below we have denoted $(u(t),u_t(t);\phi(t),\phi_t(t))=S_t(y_0)$.) Additionally, by the global-in-time bound given in Lemma \ref{globalbound}, we know that the set $\{S_{t}(y_0)\}$ is bounded in $\widetilde Y$, and hence there for any sequence $t_{n} \to \infty$ there exists a subsequence $t_{n_k}$, and a point $(\hat u, \hat w; \hat \phi, \hat \psi) \in \widetilde Y$ such that $S_{t_{n_k}}(y_0) \rightharpoonup  (\hat u, \hat w; \hat \phi, \hat \psi)$ in $Y_{\rho}$ for any $\rho$. Utilizing these results in conjunction, for any sequence of times $t_{n} \to \infty$ there is a subsequence---which we identify simply as $t_n$---such that both of the above convergences hold.

We now collect these various convergences for the plate dynamics, and use them to show that place velocities decay to zero.
\begin{theorem}\label{convergenceprops}
For any initial data $y^0=(u^0,u^1,\phi^0,\phi^1)\in Y$ and any sequence of times $t_n \to \infty$ there is a subsequence $t_{n_k}$ and a point $\hat y = (\hat u,\hat w;\hat \phi, \hat \psi)$ so that:
 \begin{enumerate}
 \item $S_{t_{n_k}}(y^0) \rightharpoonup  \hat y$. \vskip.2cm
\item
$\| u(t_{n_k}) - \hat{u}\|_{2,\Omega} \rightarrow 0$. \vskip.2cm
\item $\|u_t(t)\|^2_{0} \to 0, ~t \to \infty$, and hence $\hat w =0$. \vskip.2cm
\item Along the sequence of times $t_{n_k} \to \infty$ we have \begin{align}
\sup_{\tau \in [-c,c]}\|u(t_n+\tau)-\hat u\|_{2-\delta,\Omega} \to 0 ~\text{ for any fixed}~ c > 0, ~\delta\in (0,2).
\end{align}
\end{enumerate}
\end{theorem}
\begin{proof}
Owing to the boundedness of the dissipation integral, we now consider the quantity $\lb u_{tt}(t),\eta \rb_{\Omega}$ for $\eta \in C_0^{\infty}(\Omega)$.
\begin{lemma}\label{boundder}
Taking $u$ to be the plate component of a generalized solution to \eqref{flowplate}, the quantity $\partial_t\langle u_t(t),\eta \rangle_{\Omega}$ is uniformly bounded for $t \in [0,\infty)$ for each $\eta \in C_0^{\infty}(\Omega)$.
\end{lemma}
\begin{proof}
We note that \begin{align*}
|\lb u_{tt},\eta \rb|= &~ \big|-k\lb u_t,\eta \rb-\lb\Delta  u, \Delta \eta \rb+\lb f(u),\nabla \eta \rb+\lb p_0,\eta \rb +\lb tr[\phi_t+U\phi_x],\eta \rb \big|\\
 \le &~ \|\eta \|_{2,\Omega}\big(\mathbf E(t)+C\big)+|\lb tr[\phi_t+U\phi_x],\eta \rb_{\Omega}|
\end{align*}
We now enforce our hypotheses that $0 \le U <1$ and that there exists a $\rho_0$ such that for $|\xb|>\rho_0$~~ $\phi_0(\xb)=\phi_1(\xb)=0$. Utilizing the reduction result Theorem \ref{rewrite}, we note that for sufficiently large times $t>t^\#(\rho_0,\Omega,U)$ we may utilize the formula for the trace of the flow, namely:
$$tr[\phi_t+U\phi_x]=-(u_t+Uu_x)-q^u(t), ~~\xb \in \Omega.$$ It is then easy to see that
$$|\lb tr[\phi_t+U\phi_x],\eta\rb_{\Omega}| \le \|\eta\|_{0,\Omega}\cdot \sup_{\tau \in [0,\infty)}\left\{ \|\Delta u\|^2+\|u_t\|^2\right\} \le C\|\eta\|,$$ where in the final step we have used the boundedness in Theorem \ref{globalbound}. Hence the quantity is bounded {\em uniformly} in $t$ . \end{proof}
\noindent Now, from Corollary \ref{dissint} and Lemma \ref{boundder} above we can conclude that $$k\int_0^{\infty}|\lb u_t(\tau),\eta\rb|^2 d\tau < \infty$$ and also that $k\dfrac{d}{dt}\lb u_t(\tau),\eta\rb_{\Omega}$ is uniformly bounded in time for $t \in [0,\infty)$. Then, via the Barbalat Lemma, we can conclude that \begin{equation}\label{needed1} \lim_{t \to + \infty} k\lb u_t,\eta \rb  = 0 \end{equation} for each $\eta \in C_0^{\infty}(\Omega)$.


The weak convergence of $u_t$ to 0 in \eqref{needed1} and the strong convergence on subsequences $u_t(t_n) \rightarrow \hat w$   imply that  $\hat w =0$. Since every sequence has a convergent subsequence---which must converge strongly to zero---this proves
 that $$\int_{\Omega}|u_t(t)|^2 d\Omega \to 0, ~t \to \infty.$$ as desired for the third item in Theorem \ref{convergenceprops}

 Now, as described above, from the existence of the attracting set for  the plate component we conclude strong convergence   of
\begin{equation}\label{one} \| u(t_n) -  \hat{u}  \|_{2,\Omega}  \rightarrow 0\end{equation} when $t_n \rightarrow \infty $.

As for the lower order term, the following bound is clear
\begin{equation}\label{two}
\|u(t_n+\epsilon)-u(t_n)\|_{0,\Omega} \le \int_{t_n}^{t_n+\epsilon} \|u_t(\tau)\|_{0,\Omega}d\tau \le \epsilon\cdot \sup_{\tau \in [t_n,t_n+\epsilon]}\|u_t(\tau)\|_{0,\Omega}
\end{equation}
Thus \begin{align}
\sup_{\tau \in [-c,c]}\|u(t_n+\tau)-\hat u\|_{0,\Omega}\le& ~\sup_{\tau \in[-c,c]} \|u(t_n+\tau)-u(t_n)\|_{0,\Omega}+\|u(t_n)-\hat u\|_{0,\Omega}\\ \nonumber  & \to 0 ~\text{ for any fixed}~ c > 0
\end{align} by \eqref{one} and \eqref{two} above. More is true: by the interpolation inequality
$$\|u(t)\|_{2-\delta,\Omega} \le \|u(t)\|^{\delta/2}_{0,\Omega}\|u(t)\|^{1-\delta/2}_{2,\Omega},$$
and the boundedness of $\left\{\|u(t)\|_{2,\Omega}~:~t \ge 0 \right\}$ we see that \begin{align}
\sup_{\tau \in [-c,c]}\|u(t_n+\tau)-\hat u\|_{2-\delta,\Omega} \to 0 ~\text{ for any fixed}~ c > 0, \delta\in (0,2).
\end{align}
\subsection{Strong stability for regular data---proper proof of Theorem \ref{regresult}}
We consider an initial datum $y_m \in \cD(\bT)$. We want to show that the trajectory $S_t$ is strongly stable in some sense. \begin{theorem}
Consider $y_m \in \cD(\bT)\cap B_R(Y) \subset (H^4\cap H_0^2)(\Omega)\times H_0^2(\Omega) \times H^2(\realsthree_+) \times H^1(\realsthree_+) \subset Y$. Suppose that the initial flow data are supported on a ball of radius $\rho_0$ (as in Theorem \ref{rewrite}). Suppose $k>0$. Then for any sequence of times $t_n \to \infty$ there is a subsequence of times $t_{n_k}$ identified by $t_k$ and a point $\hat y =(\hat u, 0; \hat \phi, 0) \in \widetilde Y$ so that $$\|S_{t_k}(y_m) - \hat y\|_{Y_{\rho}} \to 0,~~t_k \to \infty.$$ \end{theorem}
\begin{remark}\label{regsize} We pause to point out that this result does not depend on the size of the damping coefficients $k,\beta$. \end{remark}
\begin{proof} Owing to Remark \ref{regsize} for this proof we consider $\beta=0$ and $k>0$ fixed.
The proof follows through four steps.
\vskip.3cm
\noindent {\em STEP 1: Uniform-in-Time Bounds  of Higher Energy }
\vskip.1cm
\noindent In order to make use of the flow compactness criterion in Lemma \ref{compactnesscriterion} we must first utilize Lemma \ref{compact2}. This requires showing an interpolation type bound on plate solutions in higher norms. In the accounts \cite{ryz,ryz2} this was accomplished using parabolic effects; here we show that such a bound can be obtained on regular data by utilizing the structure of the von Karman bracket. We will ultimately prove the following bound:
\begin{theorem}\label{highernorms}
Consider initial data $y_m \in \cD(\bT)$ such that $y_m \in B_R(Y)$ and take $k>0$. Then we have that for the trajectory $S_t(y_m) = (u_m(t),u_{mt}(t);\phi_m(t),\phi_{mt}(t))$ $$(u_m(\cdot),\phi_m(\cdot)) \in  C^1\left(0,T;H_0^2(\Omega)\times H^1(\realsthree_+)\right),$$ for any $T$,
along with the bound
\begin{equation}\label{keybound0}
\sup_{t\in[0,T)}\left\{\|\Delta u_{mt}\|^2+\|u_{mtt}\|^2 \right\} < \infty.
\end{equation}
Additionally, if we assume the flow initial data are localized and consider the delayed plate trajectory (via the reduction result Theorem \ref{rewrite}), and we take $k>0$, we will have:
\begin{equation}\label{keybound1}
\sup_{t\in[0,\infty)}\left\{\|\Delta u_{mt}\|^2+\|u_{mtt}\|^2 \right\} < \infty.
\end{equation}
Then, by the boundedness in time of each of the terms in the \eqref{reducedplate}, we have
$$\|\Delta^2u_m\|_0 \le C(R,p_0,F_0),$$ where we critically used the previous bound in \eqref{keybound1}. In particular, this implies that, taking into account the clamped boundary conditions, \begin{equation}\label{keybound2}
\sup_{t \in [0,\infty)}\|u_m(\tau)\|_{4,\Omega} \le C(R,p_0,F_0)
\end{equation}

\end{theorem}

The proof of this Theorem \ref{highernorms} is based on two sub-steps labeled STEP 1.1 and STEP 1.2. We first prove finite time propagation of regularity (STEP 1.1). Following this, we shall propagate regularity
for all times uniformly (STEP 1.2) without making assumptions on the damping coefficients. This is motivated by the fact that for the infinite time propagation we must use the delay representation for plate solutions, which is valid only for sufficiently large times. Thus, the regular initial condition required to do so is obtained via the finite time propagation of regularity for the full flow-plate model.

\begin{remark}
It is noteworthy that while finite time propagation of the regularity is valid for the full flow-plate trajectory, the infinite horizon propagation is valid only for the plate.
\end{remark}

\begin{proof}[Proof of Theorem \ref{highernorms}]
 \noindent {\em STEP 1.1}
 \vskip.1cm
To prove \eqref{keybound1} above (which then implies \eqref{keybound2}) we will consider the time differentiated version of the entire flow-plate dynamics in \eqref{flowplate}, which is permissible in $\cD(\bT)$. We label $w=u_{mt}$ and $\Phi=\phi_{mt}$, multiply the plate equation by $w_t$ and the flow equation by $\Phi_t$, and integrate in time. We consider the `strong' energies
\begin{align}
E^s_{pl}(t)=& ~\dfrac{1}{2} [\|\Delta w(t)\|^2+\|w_t(t)\|^2],\notag\\
E^s_{fl}(t)=& \dfrac{1}{2}[\|\nabla \Phi(t)\|^2_{\realsthree_+}-U^2\|\partial_x \Phi(t)\|_{\realsthree_+}^2+\|\Phi_t(t)\|^2_{\realsthree_+}],\notag\\
E^s_{int}(t)=&~U\lb w_x(t),\Phi(t)\rb_{\Omega}\notag\\
\cE^s(t) = &~E^s_{pl}(t)+E^s_{fl}(t)+E^s_{int}(t),~~
\bE^s(t) =  E^s_{pl}(t)+E^s_{fl}(t)
\end{align}
From the standard analysis (done at the energy level) we have that \cite{springer,webster}
\begin{align}
|E^s_{int}(t)| \le\notag&\\ \delta \|\nabla \phi_{mt}(t)\|^2&+\frac{C}{\delta}\|u_{mxt}(t)\|_0^2 \le \delta \|\nabla \phi_{mt}(t)\|^2+\epsilon \|\Delta u_{mt}(t)\|^2+C_{\epsilon,\delta}\|u_{mt}(t)\|^2
\end{align}
Using the boundedness of trajectories at the energy level $\sup_{t>0}\|u_{mt}(t)\|_0 \le C_R$, we then have that for all $\epsilon>0$
\begin{equation}
|E^s_{int}(t)| \le \epsilon\bE^s(t)+K_{\epsilon},
\end{equation}
and this results in the existence of constants $c,C,K>0$ such that \begin{equation}\label{equiv}
c\bE^s(t)-K \le \cE^s(t) \le C\bE^s(t)+K, ~~\forall t>0.
\end{equation}
Then the energy balance for the time differentiated equations is:
\begin{equation}\label{strongbalance}
\cE^s(t) +k\int_0^t \|w_t(\tau)\|^2 d\tau = \cE^s(0)+B(t),
\end{equation}
where $B(t)$ represents the contribution of the nonlinearity in the time-differentiated dynamics:
\begin{equation}
B(t) \equiv \int_0^t \left\lb \dfrac{d}{dt}(f(u_m)),w_t\right\rb_\Omega d\tau
\end{equation}
Using \eqref{equiv}, we have
\begin{equation}
\bE^s(t) \le C\bE^s(0)+K+|B(t)|
\end{equation}
We note immediately that if we can bound this term (perhaps using the LHS), we will obtain the estimate in \eqref{keybound1}. Continuing, we use the key decomposition (and symmetry) of the bracket given in \cite[p. 517]{springer}:
\begin{align}
\left\lb\dfrac{d}{dt}(f(u_m)),w_t \right\rb_{\Omega} = &~\lb [u_{mt},v(u_m)+F_0],u_{mtt}\rb+2\lb[u_m,v(u_m,u_{mt})],u_{mtt}\rb\notag\\
= & \dfrac{d}{dt}Q_0(t)+P_0(t)
\end{align}
where
\begin{align}
Q_0(t) = &-\|\Delta v(u_m,u_{mt})\|^2+ \frac{1}{2}\lb [u_{mt},v(u_m)+F_0],u_{mt}\rb_{\Omega}\\
P_0(t) = & -3\left \lb [u_t,v(u,u_t)],u_t\right\rb_{\Omega}.
\end{align}
Using the key property of the von Karman nonlinearity (resulting from the {\em sharp regularity of the Airy stress function} \cite{springer}):
\begin{align}
\|[w_1,v(w_2,w_3)]\|_0 \le C\|w_1\|_{2}\|w_2\|_2\|w_3\|_{2}, ~~
\end{align}
\begin{align}
|B(t)| \le &~ C\Big\{\|[u_{mt},v(u_m)+F_0]\|_0\|u_{mt}\|_0\Big|_0^t+\int_0^t\|[u_{mt},v(u_m,u_{mt})]\|_0\|u_{mt}\|_0 d\tau \Big\}\notag\\
\le & ~C_R\|u_{mt}(t)\|_2\|\|u_{mt}(t)\|_0+C\bE^s(0) + C_R\int_0^t\|u_{mt}\|_2^2\|u_{mt}\|_0 d\tau\\
\le &~\epsilon ||u_{mt}(t)||_2^2 +C_{\epsilon}(R)||u_{mt}||^2+C\bE^s(0) + C_R\int_0^t\|u_{mt}\|_2^2 d\tau
\end{align}
We again appeal to the hyperbolic estimate:
$$||u_{mt}(t)||^2_0 \le ||u_{mt}(0)||^2_0 +\int_0^t||u_{mtt}(\tau)||_0^2 d\tau.$$ Hence the above inequality, when used in conjunction with \eqref{eident}  and Gronwall's inequality, provide uniform boundedness of the higher norm energy $\bE^s(t)$ on any finite time interval $[0,T]$, for  any {\it finite }  $T$. \vskip.3cm
 \noindent {\em STEP 1.2}
 \vskip.1cm
 \begin{remark}
Using the variation of parameters formula, and global-in-time bounds of delayed plate trajectories in the energy topology, one can obtain the bounds in \eqref{keybound1} and \eqref{keybound2} straightforwardly, however this requires {\em large viscous damping}. We show below that these bounds (and the final result for strong stability for smooth initial data) can be obtained with minimal viscous damping and zero static damping. 
\end{remark}
 In order to obtain uniform-in-time boundedness for {\em plate solutions} (as in \ref{keybound1} and \ref{keybound2}), more work is needed. The needed argument depends on the reduction principle in Theorem \ref{rewrite}, where the flow dynamics are reduced
to a non-dissipative and a delayed term in the plate dynamics. This reduction is valid for sufficiently large times---thus the result of STEP 1.1 applies to all the times before the delayed model is valid.

We return to the delay representation of the  plate solutions (Theorem \ref{rewrite}) and utilize our finite horizon regularity result to obtain regular plate data at time $T$. Denoting by $(w(t),w_t(t) )  = (u_{mt},u_{mtt})$, we write
\begin{equation}\label{w}
w_{tt} + \Delta^2 w + (k + 1) w_t +\beta w =\frac{d}{dt} \left\{[ u, v(u) + F_0 ] \right\}- U \partial_x w - \frac{d}{dt} \left\{q(u)\right\}
\end{equation} which is valid for any $t>\max \{T,t^{\#}\}$.
We utilize an approach taken in \cite{springer} for the non-delayed plate. We will define a Lyapunov-type function for the time-differentiated, smooth dynamics so that we may use Gronwall, along with the finiteness of the dissipation integral (Corollary \ref{dissint}):
$$\int_0^{\infty}||u_{mt}||^2 d\tau = \int_0^{\infty} ||w||^2d\tau<+\infty.$$ 

\noindent Let us use some more convenient notation: $E(t)=\dfrac{1}{2}[||\Delta w||^2+||w_t||^2]$. Then define $$Q(t) = E(t)-Q_0(t)+(\nu+\beta)||w||^2,$$
where as above
\begin{align}
Q_0(t) =&~-\|\Delta v(u_m,u_{mt})\|^2+ \frac{1}{2}\lb [u_{mt},v(u_m)+F_0],u_{mt}\rb_{\Omega}\\
=&~-\|\Delta v(u_m,w)\|^2+ \frac{1}{2}\lb [w,v(u_m)+F_0],w\rb_{\Omega}\\
P_0(t) = &~ -3\left \lb [u_{mt},v(u_m,u_{mt})],u_{mt}\right\rb_{\Omega} \\
=&~-3\left \lb [w,v(u_m,w)],w \right\rb_{\Omega}\\
\left\langle \dfrac{d}{dt}[u_m,v(u_m)+F_0], w_t\right\rangle =&~ \dfrac{d}{dt}Q_0(t) + P_0(t)
\end{align}
Recalling the sharp regularity the Airy stress function:
$$ \| [v(u_1,u_2), u_3] \| \leq C \|u_2\|_2\|u_3\|_2 \|u_1\|_2$$
we directly have:
\begin{align}
\left\langle \dfrac{d}{dt}[u_m,v(u_m)+F_0], w_t\right\rangle \le & ~\Big(||~ [w, v(u_m)+F_0] ~||_0+2||~[u_m,v(u_m,w)]|| \Big)||w_t|| \\
\le &~C(R,F_0)||w||_2||w_t||,
\end{align}
which will be used later.

Recall that $\beta\ge 0$ is fixed for this portion of the calculation. $\nu,\mu >0$. Via the sharp regularity quoted above, we can choose $\nu(R)$ large enough so that $Q(t)$ is positive (keeping $\beta$ fixed) \cite{springer}; we make this choice and then keep $\nu$ fixed. 
We will now construct a Lyapunov-like function for the dynamics $T^w_t(x_0) = (w(t),w_t(t);w^t) \in \mathbf H$ with $x_0 \in \mathbf H$. There are constants $a_0(\nu)$ and $a_1(\nu)$ so that:
$${a_0}\Big[||\Delta w||^2+||w_t||^2\Big]\le  Q(t) \le {a_1}\Big[||\Delta w||^2+||w_t||^2\Big].$$
The parametric Lyapunov function is then:
$$W(t) = Q(t)+\varepsilon (w_t,w)+\mu\int_0^{t^*}\int_{t-s}^tE(\tau)d\tau ds,~~\varepsilon \in (0,\varepsilon_0).$$
The parameter $\mu$ will be addressed below.
The value of $\varepsilon_0$ is chosen so that 
\begin{equation}\label{abovebelow} \dfrac{1}{2}a_0 E_{\beta}(t) \le W(t) \le 2a_1E_{\beta}(t)+\mu t^*\int_{t-t^*}^tE(\tau),~~\forall~~\varepsilon \in (0,\varepsilon_0),\end{equation}
where $E_{\beta}(t) \equiv \dfrac{1}{2}\left[||\Delta w||^2+||w_t||^2+\beta ||w||^2\right]$, and we note that 
$$\mu \int_0^{t^*}\int_{t-s}^t E(\tau) d\tau ds \le \mu t^* \int_{t-t^*}^t E(\tau).$$
We compute $W'(t)$:
\begin{align}
\dfrac{d}{dt} W(t) =&~ \langle \Delta w,\Delta w_t\rangle +\langle w_t,w_{tt}\rangle -\dfrac{d}{dt}Q_0(t)+(\nu+\beta)\langle w,w_t\rangle  \\
&+\varepsilon \langle w_{tt},w\rangle +\varepsilon ||w_t||^2 +\mu t^*E(t)-\mu \int_{t-t^*}^t E(\tau)
\end{align}
Make the replacement:
$$w_{tt} = - \Delta^2 w-k_0w_t-\beta w-\dfrac{d}{dt}q^{u_m}-\dfrac{d}{dt}f(u)-Uw_x.$$
Then:
\begin{align*}
\dfrac{d}{dt} W(t) =&~(\mu t^*+\varepsilon -k_0)||w_t||^2+(\nu-\varepsilon k_0)\langle w,w_t\rangle -\langle [q^{u_m}]',w_t\rangle -\langle [f(u)]',w_t\rangle -U\langle w_x,w_t\rangle -\dfrac{d}{dt}Q_0 \\
&(\mu t^*-\varepsilon)||\Delta w||^2-\varepsilon\beta||w||^2+\varepsilon\langle [q^{u_m}]',w\rangle -\varepsilon\langle [f(u)]',w\rangle -\varepsilon U\langle w_x,w\rangle -\mu \int_{t-t^*}^t E(\tau) \end{align*}

Thus:
\begin{align}
\dfrac{d}{dt}W(t)&+k_0||w_t||^2+\varepsilon||\Delta w||^2+\varepsilon \beta ||w||^2+\mu \int_{t-t^*}^t E(\tau)d\tau  \\\nonumber
=&~(\mu t^*+\varepsilon)||w_t||^2+\mu t^* ||\Delta w||^2+P_0(t)-\varepsilon\langle [f(u)]',w\rangle \\\nonumber
&-{\varepsilon\langle [q^{u_m}]',w\rangle-\langle [q^{u_m}]',w_t\rangle }+(\nu-\varepsilon k_0)\langle w,w_t\rangle -\varepsilon U\langle w_x,w\rangle-U\langle w_x,w_t\rangle ,
\end{align}
Via \eqref{abovebelow}, choosing $\mu$ and $\varepsilon$ sufficiently small, there is a $\gamma(a_1,\nu,\beta,k_0,\mu,t^*)>0$ so that for any $\varepsilon \in (0,\varepsilon_0)$:
\begin{align}\label{zzz}
\dfrac{d}{dt}W(t)&+\gamma [E_{\beta}(t)+\int_{t-t^*}^t E(\tau) d\tau ] \le  C\Big\{ P_0(t)-\varepsilon\langle [f(u)]',w\rangle -{\varepsilon\langle[q^{u_m}]',w\rangle-\langle [q^{u_m}]',w_t\rangle}
\\\nonumber&+(\nu-\varepsilon k_0)\langle w,w_t\rangle-\varepsilon U\langle w_x,w\rangle -U\langle w_x,w_t\rangle\Big\}
\end{align}
\begin{lemma}\label{lyapo}
There exists a $\widetilde \gamma>0$ (with the same dependencies $\gamma$  above), and constant $M$ (depending on $R$ and the same quantities as $\widetilde \gamma$) so that:
\begin{equation}
W'(t)+\widetilde \gamma W(t)  \le M[1+||w||^2E(t)]
\end{equation}
\end{lemma}
\begin{proof}[Proof of Lemma \ref{lyapo}]

We recall: 
\begin{align}
q_t^{u_m}(t) =&\int_0^{2\pi}\frac{1}{2 \pi}[M^2_{\theta}\widehat u_m]\big(\xb(U,\theta,0),t \big)d\theta\nonumber\\
&-\int_0^{2\pi}\frac{1}{2 \pi}[M^2_{\theta}\widehat u_m]\big(\xb(U,\theta,t^*),t-t^* \big)d\theta\nonumber\\
&+\Big( \int_0^{t^*}\int_0^{2\pi}(U+\sin\theta)\frac{1}{2 \pi}[M^2_{\theta}\widehat u_m]_x\big(\xb(U,\theta,s),t-s \big)d\theta ds\Big)\nonumber\\
&+\Big(\int_0^{t^*}\int_0^{2\pi}(\cos\theta)\frac{1}{2 \pi}[M^2_{\theta}\widehat u_m]_y\big(\xb(U,\theta,s),t-s \big)d\theta ds\Big).
\end{align}

So, along with the global-in-time priori bounds for the lower energy energy level $(u,u_t)$,
$$\sup_{t \in [t_0,\infty)} \Big[||u_m||^2_2+||w||^2 \Big]<\infty,$$
we obtain:
\begin{align}
|P_0| \le &~C(1+||\Delta u||)||w||||\Delta w||^2 \\[.2cm]
 |\langle [f(u)]',w\rangle | \le &~ C(R)||w||||\Delta w|| \\[.2cm]
 (\nu-\varepsilon k_0)|\langle w, w_t\rangle | \le &~\delta ||w_t||^2+C_{\delta}(\nu,\varepsilon,k_0,R)\\[.2cm]
 \varepsilon U|\langle w_x,w\rangle+U|\langle w_x,w_t\rangle | \le &~\delta[ ||w_t||^2+||\Delta w||^2] +C_{\delta}(\varepsilon,U,R)\\[.2cm]
\varepsilon |\langle [q^{u_m}]',w \rangle| \le&~ C(\varepsilon)\left(||u_m(t)||^2_{2}+||u_m(t-t^*)||^2_{2}+\int_{-t^*}^0||u_m(t+\tau)||^2_2 d\tau\right) \\\nonumber
&+\delta ||w||_2^2+C_{\delta}(\varepsilon) ||w||^2 \\[.2cm]
\le &~\delta E(t) +C_{\delta}(\varepsilon,t^*,R)
\end{align}
And, noting the explicit estimates for the term $q_t^u$ performed in \cite[(6.1)]{delay}, we also have
\begin{eqnarray}
\langle \partial_t q^{u_m}, w_t \rangle \leq \delta ||w_t||^2 +C_{\delta} \left(\|u(t)\|^2_{2,\Omega}
 + \| u(t-t^*)\|^2_{2,\Omega} +
  \int_{-t^*}^0 \|u(t+\tau)\|^2_{3,\Omega} ) d \tau\right).
\end{eqnarray}

\noindent Then we note: $$\|u\|^2_{3} \leq \delta \|u\|^2_4 + C_{\epsilon}(R)\|u\|^2_2 \leq C_{\epsilon}  + \epsilon \|u\|^2_4, $$
where $\|u\|_4 $ is to be calculated from an elliptic equation in terms of $\|W(t)\|_{{Y}_{pl}}$.
Indeed, we consider biharmonic problem with  the clamped boundary conditions.
$$ \Delta^2 u_m = - w_t - (k+1) w  + [u_m, v(u_m)+F_0 ] -Uu_{mx} - q^{u_m}-\beta u_m$$
with
$$u = \Dn u =0 ~\text{on}~ \partial \Omega.$$
This gives
\begin{align*}\|u_m\|_4 \leq& ~C  \|w_t - k_0 w  + [u_m, v(u_m) + F_0 ] -u_{mx}-\beta u_m - q^{u_m}\|
\\
\leq&~ C  \left[\|(w,w_t)\|_{Y_{pl}} +{ \|u_m\|_1 + \|u_m\|^3_2 + \|u_m\|^2_2 }\right].
\end{align*}
 We have
$$\|u_m\|^2_4 \leq  C(R) [E(t) + {1}].$$

\begin{equation}
\langle [q^{u_m}]',w_t \rangle \le ||w_t||_0 \left|\left|\dfrac{d}{dt} q^{u_m} \right|\right|_0 \le \delta \left[E(t)+\int_{t-t^*}^t E(\tau) d\tau\right] + C_{\delta}(t^*,R)
\end{equation}

Combining these estimates, we return to  \eqref{zzz}:
\begin{align}
\dfrac{d}{dt}W(t)+\gamma [E_{\beta}(t)+\int_{t-t^*}^t E(\tau) d\tau]=&~  C\Big\{ P_0(t)-\varepsilon([f(u)]',w)-{\varepsilon([q(u^t)]',w)-([q(u^t)]',w_t)}
\\\nonumber&+(\nu-\varepsilon k_0)(w,w_t)-\varepsilon U(w_x,w)-U(w_x,w_t)\Big\} \\
\le &~ \delta [E(t)+\int_{t-t^*}^t E(\tau)d\tau] +C_{\delta}(\varepsilon,\nu,R) ||w||^2E(t) \\\nonumber
&+ C_{\delta}(R,\nu,\varepsilon,t^*,U,k_0)
\end{align}
Scaling $\delta$ appropriately and absorbing terms, we have:
$$\dfrac{d}{dt}W(t)+\gamma [E_{\beta}(t)+\int_{t-t^*}^t E(\tau) d\tau] \le M+C||w||^2E(t).$$
Recalling \eqref{abovebelow}, we have some $\widetilde \gamma>0 $ such that:
$$\dfrac{d}{dt}W(t)+\widetilde \gamma W(t) \le M+C||w||^2E(t).$$
\end{proof}
Having completed the proof of Lemma \ref{lyapo}, we note \eqref{abovebelow}. This gives:
$$\dfrac{d}{dt}W(t)+\widetilde \gamma W(t) \le M+C_1||w||^2 W(t),$$
and utilizing the integrating factor we see that
$$\dfrac{d}{dt}\left(e^{\widetilde \gamma t}W(t) \right) \le Me^{\widetilde \gamma t}+C_1e^{\widetilde \gamma t}||w||^2W(t).$$
Integrating, we have:
$$e^{\widetilde \gamma T} W(T) \le e^{\widetilde \gamma t_0}W(t_0)+\dfrac{M}{\widetilde \gamma}\left[e^{\widetilde \gamma T}-e^{\widetilde \gamma t_0}\right]+C_1\int_{t_0}^Te^{\widetilde \gamma t}||w(t)||^2W(t)dt,$$
and thus
\begin{align*} W(T) \le&~ e^{-\widetilde \gamma(T-t_0)} W(t_0)+\dfrac{M}{\widetilde \gamma}\left[1-e^{-\widetilde \gamma (T-t_0)}\right]+C_1\int_{t_0}^T||w(t)||^2W(t)dt \\
\le &~\dfrac{M}{\widetilde \gamma}+e^{-\widetilde \gamma(T-t_0)} \left[W(t_0)-\dfrac{M}{\widetilde \gamma}\right]+C_1\int_{t_0}^T||w(t)||^2W(t)dt.
\end{align*}
 Let $$\alpha(T) = \dfrac{M}{\widetilde \gamma}+e^{-\widetilde \gamma(T-t_0)} \left[W(t_0)-\dfrac{M}{\widetilde \gamma}\right],$$ and note that $\alpha(T) \ge 0$ for $T\ge t_0$. Then, using Grownall, we have:
$$W(T) \le \alpha(t)+\int_{t_0}^T \alpha(t)||w(t)||^2\exp\left\{\int_t^T||w(s)||^2ds\right\} dt.$$ Note that $\alpha(T) \ge 0$ and $\alpha$ is bounded, and thus, owing to the finiteness of the dissipation integral $$\int_{t_0}^{\infty}||w||^2 dt=\int_{t_0}^{\infty} ||u_{mt}||_0^2dt<+\infty,$$ we obtain global-in-time boundedness of $W(t)$. Then, using \eqref{abovebelow}, we have global-in-time boundedness of $$E_{\beta}(t)=\dfrac{1}{2}\big[||\Delta u_{mt}||^2+||u_{mtt}||^2+\beta ||u_{mt}||^2\big].$$Via elliptic theory (as utilized above), we also have 
$||u_m(t)||_4^2 $ globally bounded in time. We note that $W(t_0)$ has the requisite regularity by STEP 1.1 above (where we have shown propagation of regularity on the finite time horizon for the full flow-plate dynamics).

This concludes the proof of Theorem \ref{highernorms} and STEP 1.2
\end{proof}

 \noindent {\em STEP 2: Flow Compactness for Smooth Initial Data}
 \vskip.1cm
 \noindent We first note at this point that the solution (via classical scattering theory \cite{lax}) to \eqref{flow1} is stable, so $(\phi_m^*(t_n),\phi^*_{mt}(t_n)) \to (0,0)$ in the local flow energy topology $\widetilde Y_{fl}$.

 We now utilize the bounds in \eqref{keybound1} and \eqref{keybound2}, along with the bound in Lemma \ref{compact2}
 to bound, uniformly, the flow in higher norms. Specifically, we note that on the strength of Lemma \ref{highernorms}
 \begin{align}
 ~~u_m \in &~C^2(0,\infty;L_2(\Omega))\cap C^1(0,\infty;H_0^2(\Omega)) \cap C(0,\infty;H^4(\Omega)\cap H_0^2(\Omega))
 \end{align}
  This implies, in particular, that ~~$u_m\in H^2(L_2(\Omega)) \cap H^1(H^2(\Omega)) \cap L_2(H^4(\Omega)). $
 With~~  $u_m\in L_2(H^4(\Omega)) \cap H^1(H^2(\Omega)) $, via interpolation \cite{LM,interp}, we have
 \begin{equation}\label{1}
 u_m \in H^{\alpha_1}(H^{2+ 2 (1-\alpha_1) }(\Omega) ), ~~\alpha_1\in [0,1].
 \end{equation}
 Here and below we utilize the notation $ H^s(H^r(\Omega) ) $ meaning that the first regularity $H^s$ refers to time
 variable. Since time dependence in the formula representing the Neumann map spans only finite time interval of length $ t^*(\Omega,U)$---
this abbreviation  of notation should not lead to any  confusion.

 Similarly, from $\displaystyle u_{mt}\in L_2(H^2(\Omega)) \cap H^1(L_2(\Omega)) $
 we obtain
 \begin{equation}\label{2}
 u_{mt}\in H^{\alpha_2}( H^{2(1-\alpha_2)}(\Omega)),~~ \alpha_2 \in [0,1]
 \end{equation}
 To apply inequality in Lemma \ref{compact2} with $s=0,~\beta > 0$ we will take $\alpha_1 = \beta  =\alpha_2 $
 \begin{align*}
4 -2 \alpha_1 \geq &~2+ \alpha_1;~ \text{so}~~  2 \geq 3 \alpha_1 \\ \nonumber
\text{ and }& \\
  2 -2 \alpha_2 \geq &~1 + \alpha_2 ;~ \text{so}~~ 1 \geq 3 \alpha_2. \end{align*}
Hence, for
 $\beta \leq 1/3 $  we have:

 $\|\nabla \phi_m^{**}(t)\|^2_{\beta, K_{\rho} }+\|\phi_m^{**}(t)\|^2_{\beta, K_{\rho} }$ $$ \le \left\{\|u_m(t)\|^2_{H^{s+\beta}(t-t^*,t;H_0^{2+\beta}(\Omega))}+\|u_{mt}(t)\|^2_{H^{s+\beta}(t-t^*,t;H_0^{1+\beta}(\Omega))} \right\} \le C(R).$$ Then, applying the compactness criterion in Lemma \ref{compactnesscriterion} we have shown the following lemma:

 \begin{lemma}\label{limitpoint} Along some subsequence $t_n$
$$(u_m(t_n),u_{mt}(t_n);\phi_m(t_n),\phi_{mt}(t_n)) \to (\hat u, 0; \hat \phi, \hat \psi)$$ in $\widetilde Y$ as $n \to \infty$.
\end{lemma}
\vskip.3cm
\noindent {\em STEP 3: Characterization of the Flow Limit Point $\hat \psi$}
\vskip.1cm

\noindent Now, to further characterize the flow limit point, we return to the formula in Theorem \ref{flowformula}.
\begin{lemma}\label{characterize}
The limit point $\psi \in \widetilde Y$ in Lemma \ref{limitpoint} above is identified as $0$ in $L_2(K_{\rho})$ for any $\rho>0$.
\end{lemma}
The flow solution $(\phi_m^*(t),\phi_{mt}^*(t))$ tends to zero (again, we emphasize that this convergence is in the local flow energy sense); hence, we have that $\phi_m^{**}(t_n) \to \hat \phi$. To identify $\hat \psi$ with zero, we note that $u_{mt}(t) \in H^1(\Omega)$. This allows to differentiate the flow formula in Theorem \ref{flowformula} in time (as was done in \cite{springer} and \cite{ryz2}):
\begin{align}\label{thisonenow}
\phi^{**}_{mt}(\xb, t) = & \dfrac{1}{2\pi}\Big\{\int_0^{2\pi}u_{mt}^{\dag}(\xb,t,t^*,\theta) d\theta -\int_0^{2\pi} u_{mt}^{\dag}(\xb,t,z,\theta )d\theta \\\nonumber
+U\int_z^{t^*}\int_0^{2\pi}&[\partial_x u_{mt}^{\dag}](\xb,t,s,\theta) d\theta ds+\int_z^{t^*}\int_0^{2\pi}\dfrac{s}{\sqrt{s^2-z^2}}[M_{\theta}u_{mt}^{\dag}](\xb,t,s,\theta) d\theta ds\Big\} \end{align} where $$f^{\dag}(\xb,t,s,\theta)=f\left(x-\kappa_1(\theta,s,z), y-\kappa_2(\theta,s,z), t-s\right),$$ and $$\kappa_1(\theta,s,z)=Us+\sqrt{s^2-z^2}\sin \theta,~~\kappa_2(\theta,s,z) = \sqrt{s^2-z^2}\cos \theta.$$

\noindent For a fixed $\rho >0$, we multiply $\phi_{mt}^{**}(\xb,t)$ by a smooth function $\zeta \in C_0^{\infty}(K_{\rho})$ and integrate by parts in space---in \eqref{thisonenow} we move $\partial_x$ onto $\zeta$ in the third term and the partials $\partial_x,~\partial_y$ from $M_{\theta}$ onto $\zeta$ in term four.
This results in the bound
\begin{equation}\label{whatwe}
|(\phi^{**}_{mt},\zeta)_{L_2(K_{\rho})}| \le C(\rho)\sup_{\tau \in [0,t^*]}\|u_{mt}(t-\tau)\|_{0,\Omega}\|\zeta\|_{1,K_{\rho}}
\end{equation}

From this point, we utilize the fact that $u_{mt}(t) \to 0$ in $L_2(\Omega)$ (Theorem \ref{convergenceprops}), and hence $(\phi_{mt}^{**}(t),\zeta)_{K_{\rho}} \to 0$ and $\phi^*_{mt} \to 0$ in $L_2(K_{\rho})$ for any fixed $\rho>0$. This gives that $\phi_{mt}(t) \rightharpoonup 0$ in $L_2(\Omega)$, and we identify the limit point $\widehat \psi$ with $0$ in $L_2(\Omega)$.
\end{proof}

 \vskip.3cm
\noindent {\em STEP 4: Limit Points as Weak Solutions}
\vskip.1cm
In this section we show that the limit points obtained above satisfy the static problem in a weak sense.
\begin{lemma}\label{staticsols}
The pair $(\hat u, \hat \phi)$ as in Lemma \ref{limitpoint} and Lemma \ref{characterize} satisfies the stationary problem \eqref{static} in the variational sense.
\end{lemma}
We begin by multiplying the system \eqref{flowplate} by $\eta \in C_0^{\infty}(\Omega)$ (plate equation) and $\psi \in C_0^{\infty}(\realsthree_+)$ (flow equation) and integrate over the respective domains. This yields
\begin{equation}\begin{cases}
\lb u_{tt},\eta\rb +\lb \Delta u, \Delta \eta\rb +\lb u_t,\eta\rb +\lb f(u),\eta\rb  = \lb p_0,\eta\rb -\lb tr[\phi_t+U\phi_x],\eta\rb  \\
\left(\phi_{tt},\psi\right)+U(\phi_{tx},\psi)+U(\phi_{xt},\psi)-U^2(\phi_{x},\psi_x)=-(\nabla \phi, \nabla \psi)+\lb u_t+Uu_x, tr[\psi]\rb .
\end{cases}
\end{equation}

Now, we consider the above relations evaluated at the points $t_n$ (identified as a subsequence for which the various convergences above hold), and integrating in the time variable from $t_n$ to $t_n+c$. Limit passage on the linear terms is clear, owing to the main convergence properties for the plate component in the Theorem \ref{convergenceprops}. The locally Lipschitz nature of the von Karman nonlinearity allows us to pass with the limit on the nonlinear term (this is by now standard, \cite{springer}). We then arrive at the following static relations:
\begin{equation}\begin{cases}
\lb \Delta \hat u, \Delta w\rb+\lb f(u),w\rb = \lb p_0,w\rb+U\lb tr[\hat \phi],w_x\rb \\
(\nabla \hat \phi, \nabla \psi)-U^2(\hat \phi_{x},\psi_x)=U\lb \hat u_x, tr[\psi]\rb.
\end{cases}
\end{equation}
This implies that our limiting point $(\hat u,0; \hat \phi,0)$ of the sequence

\noindent $\left(u(t_n),u_t(t_n);\phi(t_n),\phi_t(t_n)\right)$ is in fact a solution (in the weak sense) to the static equations, i.e., it is a {\em stationary solution} to the flow-plate system \eqref{flowplate}. We have thus shown that any trajectory contains a sequence of times $\{ t_n\}$ such that, along these times, we observe convergence to a solution of the stationary problem.
\end{proof}
\subsection{Proof of Theorem \ref{weakth}}
We point out that item 1 in Theorem \ref{convergenceprops} guarantees that for any sequence $t_n \to \infty$ there is weak convergence  in $Y_{\rho}$ along a subsequence $t_{n_k}$ of the trajectory $S_{t_{n_k}}(y^0)$ to some point $\hat y=(\hat u, \hat w; \hat \phi, \hat \psi) \in Y_{\rho}$. Moreover, item 3 of Theorem \ref{convergenceprops} guarantees that (owing to the strong convergence) $\hat w =0$. We then note that \eqref{whatwe} from STEP 3 of the proof of Theorem \ref{regresult} above holds, and hence $\psi=0$. Since \eqref{thisonenow} holds in distribution, STEP 4  may be applied (as it is a variational argument) to move spatial derivatives onto the test function in order to conclude that $(\hat u, 0; \hat \phi, 0)$ satisfies \eqref{static} in the weak sense.
\section{Uniform-in-Time Hadamard Convergence of Semigroups---Proof of Theorem \ref{strongcont}}\label{6}

Our goal is to show Theorem \ref{strongcont}.
We begin by showing a Lipschitz estimate on any finite time interval. Consider the difference of two full flow-plate trajectories on a finite time horizon $[0,T]$. Considering initial data $y^0_m$ and $y^0$ taken from some ball of radius $R$ in $Y$, we have the following energy identity for the difference (applying the respective velocity multipliers---$(u_t-u_{mt})$ and $(\phi_t-\phi_{mt})$---to both the plate and flow equations):
 \begin{align}\label{needed}
 E_z(t) + \int_0^tk\|\bu_t\|^2 d\tau = &~E_z(0)+2U\lb \bp(0),\bu_x(0)\rb\notag\\&+\int_0^t \left\lb f_V(u)-f_V(u_m),u_t-u_{mt}\right\rb_{\Omega} d\tau -2U\lb \bp,\bu_x\rb \end{align}
 where we have utilized the notation \begin{align*}
 \bu = & ~u-u_m;~~
 \bp =  \phi-\phi_m \\
 E_z=&~ \frac{1}{2}\left(\|\Delta \bu\|^2+\|\bu_t\|^2+\beta\|\bu\|^2+\|\nabla \bp\|^2 -U^2\|\bp_x\|^2+\|\bp_t\|^2\right)
 \end{align*}
Utilizing the bounds on $E_{int}$ as in Lemma \ref{energybound}, as well as the global-in-time energy bounds (Lemma \ref{globalbound})
 one obtains
 \begin{align}
 E_z(t) + \int_0^t k\|\bu_t\|^2d\tau & \nonumber \\ \le C(U)E_z(0)+&C\|\bu_x(t)\|^2+\int_0^t\big[C(\epsilon)\|f_V(u)-f_V(u_m)\|_{0,\Omega}^2+\epsilon\|\bu_t\|_0^2 \big]d\tau
 \end{align} Via compactness $\|\bu_x\|^2 \le \epsilon \|\Delta \bu \|^2 + C(\epsilon) \|\bu\|^2$, and hence
 \begin{align}
 E_z(t) + \int_0^t k\|\bu_t\|^2d\tau & \nonumber \\ \le C(U)E_z(0)+&C\|\bu(t)\|^2+\int_0^t\big[C(\epsilon)\|f_V(u)-f_V(u_m)\|_{0,\Omega}^2+\epsilon\|\bu_t\|_0^2 \big]d\tau\end{align}
 Observing that \begin{equation}\label{oneusing} \|\bu(t)\|^2 = \left|\left|\int_0^t \bu_t(\tau)d\tau +\bu(0) \right|\right|^2 \le \int_0^t \|\bu_t\|^2d\tau +\|\bu(0)\|^2 \le \int_0^t E_z(\tau)d\tau+E_z(0)\end{equation}
 and invoking the locally Lipschitz character of $f_V$ \cite[p.44]{springer}, we have
 \begin{align}
 E_z(t) + \int_0^t k\|\bu_t\|^2d\tau \le C(U)E_z(0)+C(R)\int_0^t E_z(\tau)d\tau
 \end{align}
  (where $C(R)$ denotes the dependence of the constant on the size of the ball containing the initial datum $y^0$ and $y^0_m$).
 Gronwall's inequality then yields that
 \begin{equation}\label{gronwall} E_z(t) \le C(U,R,T)E_z(0). \end{equation}
Hence, we obtain a Hadamard continuity of the semigroups on any $[0,T]$:
 $$S_{t} (y^0_{m}) \rightarrow y(t)  \in Y,$$ for $y^0_m \to y^0 \in Y,$
 for all $t< T$, with this strong convergence depending---possibly---on $T > 0$.

We now address this continuity on the infinite-time horizon. We return to \eqref{needed}, the energy identity for the difference of two solutions (both plate and flow):
\begin{align}
E_z(t)+\int_{0}^t k\|\bu_t\|^2 d\tau \le & ~ C\Big\{E_z(0)+\|\bu(t)\|^2_1+\Big|\int_{0}^t\langle f(u)-f(u_m),\bu_t\rangle d\tau\Big| \Big\}\notag\\
\le C\Big\{E_z(0)+\epsilon E_z(t)&~ +\|\bu (t)\|^2+\int_{0}^t \| f(u)-f(u_m)\|_0\| \bu_t\|_0 d\tau \Big\}
\end{align}
Again using \eqref{oneusing}, we have:
\begin{align}
E_z(t)+\int_{0}^t k\|\bu_t\|^2 d\tau \le & ~ C\Big\{E_z(0)+\int_0^t\|\bu_t\|^2 d\tau+\int_{0}^t \| f(u)-f(u_m)\|_0\| \bu_t\|_0 d\tau \Big\}
\end{align}
We invoke the locally Lipschitz character of $f_V$:
\begin{align*}E_z(t)+\int_{0}^t k\|\bu_t\|^2 d\tau \le & ~ C\Big\{E_z(0)+\int_{0}^t\big(\|\bu_t\|^2+C(R)\|\bu_t\| \big)d\tau \Big\}\end{align*}
 Via the uniform exponential decay in Theorem \ref{exp}, and the existence of an absorbing ball $\mathscr B$, for every $\epsilon>0$ there exist a time $T^*>t^{\#}$ ($T^*$ depending on the underlying parameters of the problem $p_0,F_0,\Omega,U$, as well as the time of absorption for $\mathscr B$, and the size of the support of the initial flow data $\rho$) such that $$\int_{T^*}^{\infty}\big(\|\bu_t\|^2+C(R)\|\bu_t\|\big) d\tau \le \epsilon.$$

 Thus, for any $t>T^*$  we may write:
 \begin{align}E_z(t) \le &   C\Big\{E_z(0)\!+\!\int_{0}^{T^*}\big(\|\bu_t\|^2\!+\!C(R)\|\bu_t\|\big) d\tau\!+\!\int_{T^*}^{\infty}\big(\|\bu_t\|^2\!+\!C(R)\|\bu_t\|\big) d\tau \Big\} \end{align}
 Utilizing \eqref{gronwall}, for any $\epsilon>0$ we have:
 \begin{align}
 E_z(t) \le C(R,T^*)\big(E_z(0)+E^{1/2}_z(0)\big)+\epsilon/2
 \end{align}
 Taking $y^0_m$ sufficiently close to $y^0$ will yield that $C(R,T^*)\big(E_z(0)+E_z^{1/2}(0)\big) \le \epsilon /2$.

 This concludes the proof of Theorem \ref{regresult}.
\section{Sensitivity analysis and discussion of damping coefficients}\label{sensi}
In this final section we provide a discussion focusing on the need for the sufficiently large damping parameters $k_0$ and $\beta_0$.

First, we recall that in \cite{delay} the existence of a global attracting set for the plate dynamics can be shown in the absence of imposed mechanical damping, i.e., $k=0$ and $\beta=0$. (See the introductory discussion in Section \ref{prevv} and Theorem \ref{th:main2}.) Moreover, the size of a corresponding absorbing ball does not depend on the size of the damping parameters.

In the context of the results discussed herein, we see that Theorem \ref{regresult}---convergence to the equilibrium set for smooth initial data---is possible {\em only considering minimal damping}, i.e., for any $k>0$ and with $\beta=0$. This is in line with previous considerations where plate trajectories are ``regular" (via thermoelastic smoothing, or via rotational inertia). However, in all of these situations various notions of compactness in the plate dynamics can be easily transferred to the flow component of the dynamics. In fact, the main technical challenge in previous analyses of flow plate systems is recovering compactness (and convergence properties) in the flow {\em via the coupling}. Indeed, utilizing the reduction result in Theorem \ref{rewrite} we may obtain compact limiting behavior, as well as show $\|u_t(t)\|\to 0$ as $t\to \infty$. However, owing to the failure of the uniform Lopatinski condition for the dynamics, passing plate information (dissipation) directly to the flow is not immediately possible. The approach we have taken here is to utilize an approximation argument in the vein of \cite{ball,koch}; by considering regular initial data, we may approximate our original, finite energy initial data in $Y$. We then work to pass the convergence properties associated with trajectories emanating from regular initial data onto the original finite energy initial data. Such an argument (since it considers $t\to \infty$) is possible when the dynamics are uniformly in time Hadamard continuous.

We now show that with mild damping---any $k>0$ and $\beta=0$---we can obtain uniform-in-time convergence of semigroup norms.
 \begin{theorem}
 For a sequence of initial data $y^0_m \in Y$ such that $y^0_m \to y^0 \in Y, ~m \to \infty$ we have that
 $$\sup_{t>0}\Big|\|S_t(y^0_m)\|_{Y_{\rho}}-\|S_t(y^0)\|_{Y_{\rho}}\Big| \to 0, ~~m \to \infty,$$ where $S_t(\cdot)$ is the semigroup as in Theorem \ref{nonlinearsolution}.
 \end{theorem}

\begin{proof}
 To accomplish the task we will  be making use of the energy identity established in \cite{webster,springer,jadea12} for weak solutions.

We consider a sequence of initial data from $y^0_m \in \cD(\bT)$ such that $ y^0_m \to y_0$ strongly in $Y$. Denoting by $\mathcal E_m(t)$ the full nonlinear flow-plate energy for the trajectory $S_t(y_m^0)$ with initial datum $y^0_m \in \cD(\bT)$ (and denoting the energy for initial data $y^0 \in Y$ by $\mathcal E(t)$) we analyze the difference $\mathcal E(t) - \mathcal E_m(t)$.

Recalling the notation $\bE(t)$ (as in Lemma \ref{energybound}) we note a preliminary fact:
If $\cE_m(t) - \cE(t) \to 0$ {\em and} $y_m(t) \rightharpoonup y(t)$ in the topology of $\widetilde Y$, then $\bE_m(t) - \bE(t) \to 0$. Indeed, (suppressing $t$)
\begin{align}
\bE - \bE_m =&~\cE-\cE_m  -2U\lb \phi,u_x\rb+2U\lb \phi_m,u_{mx}\rb \nonumber \\&+\lb u-u_m,p_0\rb +\lb[u,u]-[u_m,u_m],F_0\rb
\end{align}
However, in the topology of the finite energy space (corresponding to convergence in $\widetilde Y$), the terms on lines two and three above are compact (lower order) and hence the weak convergence $S_t(y^m_0) \rightharpoonup S_t(y_0)$ is improved to strong convergence; hence these terms vanish. The von Karman bracket term in line four above vanishes via the weak continuity of the von Karman bracket. With the fact above in mind we need only show $\cE_m(t) - \cE(t) \to 0$ uniformly in time to complete the result in Theorem \ref{strongcont}.

We first let $t$ be fixed and look at \eqref{eident} applied to both $\mathcal E$ and $\cE_m$ on $[0,t]$:
\begin{align}\label{energies**}
\cE_m(t) + \int_0^tk\|u_{mt}\|^2 d\tau = \cE_m(0);~~\cE(t) + \int_0^tk\|u_{t}\|^2 d\tau = ~\cE(0)
\end{align}

From weak convergence of the trajectories and lower semicontinuity of the norm we obtain for all $ t > 0 $

\begin{equation}\label{lows}
\int_0^tk\|u_t\|^2 d\tau \leq \liminf_m \int_0^t k\|u_{mt}\|^2d\tau
\end{equation}
Then
\begin{align}\nonumber
\cE_m(t)-\cE(t) = & ~\cE_m(t) - \cE(t) +\int_0^tk\|u_{mt}\|^2d\tau-\int_0^tk\|u_{t}\|^2d\tau \\
&-\left[\int_0^tk\|u_{mt}\|^2d\tau-\int_0^tk\|u_{t}\|^2d\tau\right]
\end{align}
The above equalities \eqref{energies**} give:
\begin{align}
\cE_m(t)-\cE(t) \le &~\cE_m(0)-\cE(0)-\left[\int_0^t k\|u_{mt}\|^2 - \int_0^t k\|u_{t}\|^2\right]
\end{align}
Taking the $\limsup$ in $m$ we have
\begin{align}
\limsup_m [\cE_m(t) - \cE(t) ]\le &~\lim_m[\cE_m(0)-\cE(0)] & \nonumber \\  +\limsup_m&\left(-\int_0^tk\|u_{mt}\|^2 + \int_0^t k\|u_{t}\|^2\right)\nonumber\\
\le  ~\lim_m[\cE_m(0)&-\cE(0)] -\liminf_m\left(\int_0^tk\|u_{mt}\|^2 - \int_0^t k\|u_{t}\|^2\right)\nonumber\\
\le ~ \lim_m[\cE_m(0)&-\cE(0)] = 0
\end{align}
where we have used \eqref{lows} in the last line, uniqueness of weak solutions, and the continuity of the energy functional on $Y$ ($\cE_m(0) \to \cE(0)$). The $\limsup$ is taken uniformly in time.
Now, by the lower semicontinuity of $\cE$ on $Y$ we have
$$\sup_{t>0}[\liminf_m\cE_m(t) - \cE(t)] \ge 0.$$
Hence, we observe
$$\sup_{t>0}~ \limsup_m [\cE_m(t)-\cE(t)] \le 0 \le \sup_{t>0}~\liminf_m~[\cE_m(t) -\cE(t)],$$
and thus
\begin{align}
\sup_{t>0}\big[\lim_m\cE_m(t)-\cE(t)\big]=0.
\end{align}

The above convergence gives that---uniformly in time $t>0$---$\bE(t)-\bE_m(t) \to 0$ as $m \to \infty$, and thus $$\|S_t(y^0 )\|_{Y_{\rho}}  -  \|S_{t} (y^0_{m})\|_{Y_{\rho}}  \rightarrow 0,~~m \rightarrow \infty$$ for any $\rho>0$.
\end{proof}

In the arguments presented above, we analyzed the energy relation for the full flow-plate system. One might suspect that the dynamics are uniform-in-time Lipschitz continuous without relying on sufficiently large damping coefficients; such result is interesting from the point of view of sensitivity analysis and {\em flutter}. Uniform-in-time  Hadamard  continuity certainly holds in the linear case ($f_V=0$), or in the nonlinear case on any finite time interval (as shown above). Uniform-in-time Hadamard  continuity is also true when the attracting set for the plate dynamics is reduced to a single point (\cite{delay} or see Theorem \ref{th:main2} below). However, whether  the uniform-in-time Lipschitz property is valid in general is presently not known.

Hence, in order to obtain the uniform-in-time Hadamard continuity of the semigroup $S_t(\cdot)$ (via the energy relation on the difference of trajectories) it seems necessary to have $u_t \in L_1(0,\infty; L_2(\Omega))\cap L_2(0,\infty);L_2(\Omega)$. We accomplish this via Theorem \ref{exp} by choosing the damping coefficients $k$ and $\beta$ sufficiently large.

Another interesting possibility arises in considering that the attracting set for the plate dynamics is {\em regular}---see Theorem \ref{th:main2}. Any trajectory entering the attracting set $\mathscr U$ at some finite time obtains sufficient regularity for our results in Section \ref{5} to apply, yielding the desired end result. However, a trajectory which does not enter the attractor in finite time may converge to the attractor (perhaps exponentially fast) without providing the necessary control on the decay of the velocity for our arguments above to hold. Hence, at this point, we cannot envision a proof of the main result in Theorem \ref{conequil1} without at least assuming a sufficiently large viscous damping parameter.

The technical crux of the long-time behavior analysis for this model hinges upon the sensitivity analysis between the difference of two ``nearby" trajectories. In the analysis we have just presented, sufficiently large damping guarantees that nearby points stay nearby, and hence, showing desirable convergence properties for regular initial data can be transferred to nearby finite energy data. However, in general, for the nonlinear flow-plate problem sensitivity analysis is a key concern. Indeed, when studying flutter, ``small" perturbations in the system can produce ``large" effects in the overall dynamics. We have shown that by incorporating sufficiently large damping, we can control the end behavior of the flow-plate system---witnessing it converge to  the set of stationary states. Moreover, if we make the physical assumption that this set $\mathcal N$ is discrete, we see that the damping eliminates non-static end behavior.

\end{document}